\documentclass[a4paper,12pt,reqno]{amsart}
\usepackage{amsmath,amsfonts,amsthm,amssymb}
\usepackage{dsfont}
\usepackage[foot]{amsaddr}
\usepackage{enumitem}
\usepackage{hyperref}
\usepackage{graphicx}
\usepackage{srcltx}
\usepackage{cancel}
\usepackage{tikz}
\usetikzlibrary{matrix,shapes,arrows,positioning,chains}

\tikzset{
    desicion/.style={
        diamond,
        draw,
        text width=3em,
        text badly centered,
        inner sep=0pt
    },
    block/.style={
        rectangle,
        draw,
        text width=8em,
        text centered,
        rounded corners
    },
     block2/.style={
        rectangle,
        draw,
        text width=10em,
        text centered,
        rounded corners
    },
    cloud/.style={
        draw,
        ellipse,
        minimum height=2em
    },
    descr/.style={
        fill=white,
        inner sep=2.5pt
    },
    connector/.style={
        -latex,
        font=\scriptsize
    },
    rectangle connector/.style={
        connector,
        to path={(\tikztostart) -- ++(#1,0pt) \tikztonodes |- (\tikztotarget) },
        pos=0.5
    },
    rectangle connector/.default=-2cm,
    straight connector/.style={
        connector,
        to path=--(\tikztotarget) \tikztonodes
    }
}
\setlist[itemize]{wide=0pt, leftmargin=0pt}
\theoremstyle{plain}
\newtheorem{lem}{Lemma}[section]
\newtheorem{thm}[lem]{Theorem}
\newtheorem{prop}[lem]{Proposition}
\newtheorem{cor}[lem]{Corollary}
\theoremstyle{definition}

\newtheorem{defn}[lem]{Definition}
\newtheorem{rem}[lem]{Remark}

\newtheorem*{assc1}{Assumption $\mathbf{C^1}$}
\numberwithin{equation}{section}
\newcommand{\N}{\mathbb{N}}
\newcommand{\R}{\mathbb{R}}
\newcommand{\C}{\mathbb{C}}
\newcommand{\K}{\mathbb{K}}
\newcommand{\norm}[1]{\left\Vert#1\right\Vert}
\newcommand{\abs}[1]{\left\vert#1\right\vert}
\newcommand{\eps}{\varepsilon}
\newcommand{\diam}{\operatorname{diam}}
\newcommand{\dist}{\operatorname{dist}}
\newcommand{\dimf}{\operatorname{dim}_{f}}
\newcommand{\cl}{\operatorname{cl}}
\newcommand{\lin}{\operatorname{span}}
\renewcommand{\geq}{\geqslant}
\renewcommand{\leq}{\leqslant}
\begin{document}
\title[A panoramic view of exponential attractors]{A panoramic view of exponential attractors}

\author{Rados{\l}aw Czaja$^{1}$}

\author{Stefanie Sonner$^{2,*}$}

\address{$^1$Institute of Mathematics, University of Silesia in Katowice, Bankowa 14, 40-007 Katowice, Poland.
E-mail address: \textup{radoslaw.czaja@us.edu.pl}}

\address{$^2$Department of Mathematics, Radboud University Nijmegen, PO Box 9010, 6500 GL Nijmegen, The Netherlands.
E-mail address: \textup{stefanie.sonner@ru.nl}}

\thanks{$^{*}$ Corresponding author}


%37L30 Infinite-dimensional dissipative dynamical systems attractors and their dimensions, Lyapunov exponents
%35B41 Attractors
%37-02 Research exposition (monographs, survey articles) pertaining to dynamical systems and ergodic theory
%37L25 Inertial manifolds and other invariant attracting sets of infinite-dimensional dissipative dynamical systems

\begin{abstract}
We state necessary and sufficient conditions for the existence of $T$-discrete exponential attractors for semigroups in complete metric spaces. These conditions are formulated in terms of a covering condition for iterates of the absorbing set under the time evolution of the semigroup and imply the existence and finite-dimensionality of the  global attractor. We then review, generalize and compare existing construction methods for exponential attractors and show that they all imply the covering condition. Furthermore, we relate the results and concept of $T$-discrete exponential attractors to the classical notion of exponential attractors.
\end{abstract}

\maketitle
\section{Introduction}

Exponential attractors of infinite-dimensional dynamical systems are compact, positively invariant subsets of finite fractal dimension that attract all bounded subsets at an exponential rate. 
They contain the global attractor and hence, the existence of an exponential attractor implies the existence  of the global attractor and its finite fractal dimension. Different from global attractors, exponential attractors are not unique and there exist different methods for their construction. The first existence proof by A.~Eden, C.~Foias, B.~Nicolaenko and R.~Temam in \cite{EFNT94} was developed for semigroups in Hilbert spaces and is based on the squeezing property of the semigroup. The most general construction method by I.~Chueshov and I.~Lasiecka in \cite{cl2008} is formulated for semigroups in complete metric spaces and is based on the quasi-stability of the semigroup. We refer to the monographs \cite{EFNT94, chueshov} and the book chapter \cite{MiZe} for an overview of existence results, properties of exponential attractors and historical remarks.

In this paper we take a broader perspective, ``a panoramic view'', aiming to provide a~unifying framework for the construction of exponential attractors and to generalize, improve and compare existing, commonly used construction methods. Following the approach by D.~Pra\v{z}\'ak~\cite{Pra03b}, we formulate abstract necessary and sufficient conditions for the existence of exponential attractors for time discrete semigroups in complete metric spaces. The same characterization is possible for semigroups defined on the time interval $[0,\infty)$ if we replace the positive invariance of the exponential attractor by the positive invariance with respect to discrete time steps $T>0$. This leads to the concept of  $T$-discrete exponential attractors.   $T$-discrete exponential attractors are equivalent to classical exponential attractors in the time discrete setting. Furthermore, their construction for semigroups defined for times $t\in[0,\infty)$ does not require the H\"older continuity in time of the semigroup, as typically assumed,  which is a restrictive assumption. Our criterion for the existence of a $T$-discrete exponential attractor is formulated in terms of a covering condition for iterates of the absorbing set under the time evolution of the semigroup. The parameters in the covering condition determine the estimate for the fractal dimension of the exponential attractor and the exponential rate of attraction. Moreover, if a~$T$-discrete exponential attractor
exists, it contains the global attractor, which hence, exists and has finite fractal dimension. We also observe that the existence of a $T$-discrete exponential attractor for some time step $T>0$ implies the existence of a $\widetilde T$-discrete exponential attractor for arbitrarily small $\widetilde T>0$. 

Using our characterization of semigroups possessing a $T$-discrete exponential attractor we then verify the covering condition for widely used construction methods for exponential attractors. Generalizing previous notions and methods we can compare these different approaches. We start with the most general setting, namely quasi-stable semigroups in complete metric spaces \cite{cl2008}. We show that the quasi-stability of a semigroup implies the covering condition and hence, the existence of a $T$-discrete exponential attractor. The dimension estimate and exponential rate of attraction are determined by the parameters in the quasi-stability condition. Then, we consider semigroups in Banach spaces that satisfy a~generalized smoothing property based on the compact embedding between the phase space and another normed space \cite{CzEf}. We show that such semigroups are quasi-stable and hence, possess a $T$-discrete exponential attractor. The dimension estimate and the exponential rate of attraction are determined by embedding properties of the corresponding spaces. A subclass of these semigroups are semigroups in Banach spaces satisfying the smoothing property \cite{EfMiZe}. These semigroups can be decomposed into a sum of a~compact map and a contraction. Finally, we discuss two classes of semigroups that were originally considered in a~Hilbert space setting, namely, squeezing semigroups \cite{EFNT94} and Ladyzhenskaya type semigroups \cite{Lad82b}. We generalize the setting to Banach spaces and introduce the notion of a generalized squeezing property. Instead of an orthogonal projection onto a finite-dimensional subspace, we allow for
a possibly nonlinear map taking values in a finite-dimensional normed space.
We show that squeezing semigroups satisfy the generalized squeezing property while semigroups of Ladyzhenskaya type satisfy both, the generalized squeezing property and the smoothing property. In fact, they can also be cast into the framework of squeezing semigroups. Hence, these classes of semigroups are quasi-stable. For semigroups in Hilbert spaces we improve the estimates for the fractal dimension of the exponential attractors 
compared to the bounds obtained via quasi-stability by exploiting the Hilbert structure of the phase space and the specific properties of squeezing semigroups and Ladyzhenskaya type semigroups, respectively. Furthermore, if the phase space is a Hilbert space, we show that the class of semigroups satisfying the smoothing property coincides with the class of Ladyzhenskaya type semigroups.

The following diagrams summarize our main results. The first figure 
illustrates that quasi-stability implies the covering condition, which is equivalent to the existence of $T$-discrete exponential attractors if the semigroup possesses a bounded absorbing set:

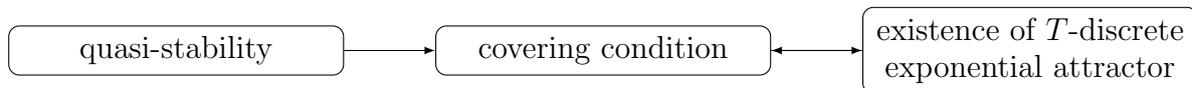
\begin{figure}[!htb]
\centering
\begin{tikzpicture}
    \matrix (m)[matrix of nodes, column  sep=0.3cm,row  sep=5mm, align=center, nodes={rectangle,draw, anchor=center} ]{
        |[block2]| {quasi-stability}         &   & &  &    |[block2]| {covering condition}   & &  &    & |[block2]| {existence of $T$-discrete exponential attractor}       \\
    };
     \path [>=latex,->] (m-1-1) edge (m-1-5);
      \path [>=latex,->] (m-1-5) edge (m-1-9);
      \path [>=latex,->] (m-1-9) edge (m-1-5);
\end{tikzpicture}
\end{figure}

The second figure 
illustrates the relations between classes of semigroups considered in commonly used  construction methods for exponential attractors:
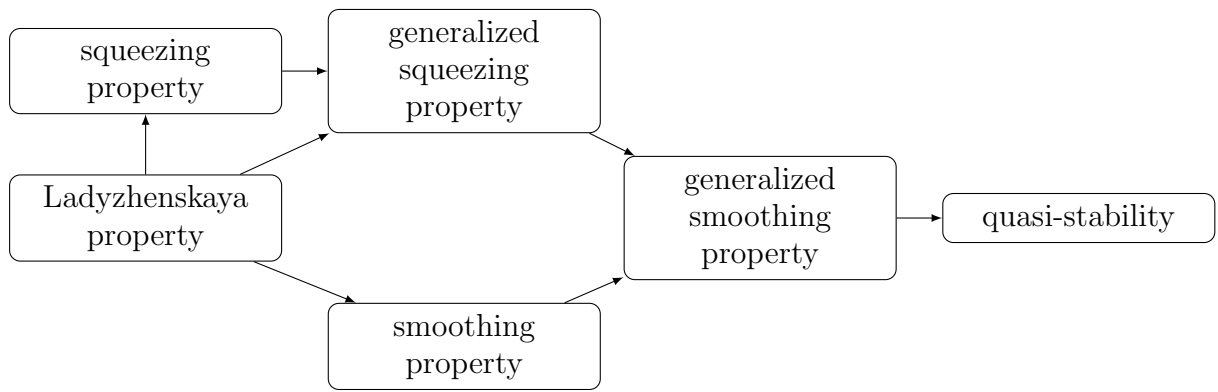
\begin{figure}[!htb]
\centering
\begin{tikzpicture}
\matrix (m)[matrix of nodes, column sep=0.3cm,row sep=3mm, align=center, nodes={rectangle,draw, anchor=center} ]{
 |[block]| {squeezing property}            &   &    |[block]| {generalized squeezing property} &    &                       &          \\
 |[block]| {Ladyzhenskaya property} &  &     & |[block]| {generalized smoothing property} &     &            |[block]| {quasi-stability}   \\
       &  & |[block]| {smoothing property}&         &                                    &\\
   };
 \path [>=latex,->] (m-1-1) edge (m-1-3);
 \path [>=latex,->] (m-1-3) edge (m-2-4);
 \path [>=latex,->] (m-2-1) edge (m-1-1);
 \path [>=latex,->] (m-2-4) edge (m-2-6);
 \path [>=latex,->] (m-3-3) edge (m-2-4);
 \path [>=latex,->] (m-2-1) edge (m-3-3);
 \path [>=latex,->] (m-2-1) edge (m-1-3);
\end{tikzpicture}
\end{figure}
 
The outline of our paper is as follows.
In Section \ref{sec:EXPO} we introduce  $T$-discrete exponential attractors and prove the existence criterion which is formulated in terms of the mentioned covering condition. Section \ref{sec:QS}  is devoted to quasi-stable semigroups in complete metric spaces. We show that quasi-stability implies the covering condition and hence, the existence of a~$T$-discrete exponential attractor.  
Section \ref{sec:CE} addresses a construction of exponential attractors for semigroups in Banach spaces based on compact embeddings, and we show that these hypotheses imply quasi-stability. 
In Section \ref{sec:SP} we discuss semigroups satisfying the smoothing property and use the results of the previous sections to conclude that such semigroups are quasi-stable. 
In Section~\ref{sec:SQ} we discuss squeezing semigroups and in Section \ref{sec:Lad} semigroups of Ladyzhenskaya type  
in a Banach space setting  and show that these semigroups are both quasi-stable. If the phase space is a Hilbert space, we further improve the estimates on the fractal dimension of the exponential attractor. 
In Section~\ref{sec:C1}, we verify the covering condition in another construction of $T$-discrete exponential attractors, which requires
the existence of a global attractor, the continuous differentiability for the semigroup and a special structure of its derivatives.  
In Section~\ref{sec:EA} we compare the notion of $T$-discrete exponential attractors for semigroups defined on time interval $[0,\infty)$ with the classical notion of exponential attractors. Using the construction in Section~\ref{sec:EXPO} and 
assuming, in addition, the H\"older continuity in time of the semigroup we provide an existence result for classical exponential attractors. Finally, in Section~\ref{sec:FC} we comment on some related notions and approaches used in the extensive literature on exponential attractors.

\section{Existence criterion for T-discrete exponential attractors}\label{sec:EXPO}

In this section we formulate abstract necessary and sufficient conditions for the existence of $T$-discrete exponential attractors for a semigroup $\{S(t)\colon t\geq 0\}$ on a metric space $(V,d)$, that is, a~family of maps $S(t)\colon V\to V$, $t\geq 0$, such that $S(t)S(s)=S(t+s)$, $t,s\geq 0$, with $S(0)=I$ being the identity map on $V$. Before we introduce the notion of an exponential attractor and its weaker counterpart of a $T$-discrete exponential attractor  we recall the concept of a global attractor, see e.g. \cite{Hale88, Lad91, C-D00}. Unless specified otherwise, when we write $t\geq 0$ then either $t\in[0,\infty)$ or  $t\in \N_0=\N\cup\{0\}$.

\begin{defn}
A \emph{global attractor} for a semigroup $\{S(t)\colon t\geq 0\}$ on a metric space $(V,d)$ is a~nonempty compact set $\mathbf{A}\subseteq V$ such that
\begin{itemize}
\item[(i)] $\mathbf{A}$ is invariant under the semigroup, i.e., $S(t)\mathbf{A}=\mathbf{A}$ for all $t\geq 0,$
\item[(ii)] for every bounded subset $G$ of $V$ we have
$$\lim_{t\to\infty}\dist^{V}(S(t)G,\mathbf{A})=\lim_{t\to\infty}\sup_{x\in G}\inf_{y\in\mathbf{A}}d(S(t)x,y)=0.$$
\end{itemize}
\end{defn}

We further recall that $\mathbf{B}\subseteq V$ is an absorbing set for the semigroup $\{S(t)\colon t\geq 0\}$ if for every bounded subset $G\subseteq V$ there exists $t_G\geq 0$ such that 
$S(t)G\subseteq \mathbf{B}$ for all $t\geq t_G$.
In existence theorems for global and exponential attractors, the semigroup is typically assumed to be continuous, or closed, see e.g.~\cite{PaZe,Pra03b,CaSo}. Here we assume asymptotic closedness of the semigroup in the main theorems on the existence of global and $T$-discrete exponential attractors (Theorems~\ref{thm:existglobal} and \ref{thm:EXDEXPAT}). Note that the asymptotic closedness of the semigroup is not required to construct $T$-discrete exponential attractor if we know a priori that a global attractor exists (Corollary~\ref{cor:WITHGLOBALATTR}). The concept of asymptotic closedness was exploited, for example, in \cite{cholewarodriguezbernal10}.

\begin{defn}
A semigroup $\{S(t)\colon t\geq 0\}$ on a metric space $(V,d)$ is called \emph{asymptotically closed}
if for any $t\geq 0$, $t_{k}\geq 0$, $t_{k}\to\infty$ and any bounded sequence $x_k\in V$
the following implication holds:
$$\text{if }S(t_k)x_k\to x\text{ and }S(t+t_k)x_{k}\to y\text{ with }x,y\in V,\  
\text{ then }S(t)x=y.$$
\end{defn}

The following theorem provides an existence criterion for global attractors. Additional equivalent characterizations can be derived, but we restrict the formulation to the statements we use in the sequel for the construction of exponential attractors.
Different from similar criteria in e.g.~\cite{chueshov}, we only assume asymptotic closedness of the semigroup. 
Here and in the sequel, $\Lambda^{V}(G)$ stands for the $\omega$-limit set of a subset $G\subseteq V$, that is, 
$$\Lambda^{V}(G)=\bigcap_{s\geq 0}\cl_{V}\bigcup_{t\geq s}S(t)G,$$
where $\cl_V$ denotes the closure in $V$.

\begin{thm}\label{thm:existglobal}
Let $\{S(t)\colon t\geq 0\}$ be an asymptotically closed semigroup on a metric space $(V,d)$. Then, the following statements are equivalent:
\begin{itemize}
\item[$(1)$] There exists a global attractor $\mathbf{A}$ for the semigroup in $V$.
\item[$(2)$] There exists a nonempty bounded absorbing set $\mathbf{B}\subseteq V$ for the semigroup such that 
for every sequence $t_k\geq 0$, $t_k\to\infty,$ and $x_k\in \mathbf{B}$, the sequence $S(t_k)x_k$ possesses a convergent subsequence. In this case, $\mathbf{A}=\displaystyle\Lambda^{V}(\mathbf{B})$ is the global attractor.
\item[$(3)$] There exists a nonempty compact set $K\subseteq V$ attracting all bounded sets of $V$. In this case, $\mathbf{A}=\displaystyle\Lambda^{V}(K)$ is the global attractor.
\end{itemize}
\end{thm}

\begin{proof}
\textbf{Step 1:} $(2)$ implies $(1)$. We first prove that the set $\Lambda^{V}(\mathbf{B})$ is a nonempty, compact subset of $V$ that attracts all bounded
subsets of $V$. To this end, consider sequences $t_k\geq 0$, $t_k\to\infty$ and $x_k\in \mathbf{B}$.
By assumption there exists $y\in V$ and a subsequence  such that
$S(t_{k_j})x_{k_j}\to y$, and by the definition of $\Lambda^{V}(\mathbf{B})$ we have $y\in \Lambda^{V}(\mathbf{B})$.
To show compactness let now $y_k\in\Lambda^{V}(\mathbf{B})$, $k\in\N$. Then, there exist  $t_{k}\geq k$ and $x_k\in \mathbf{B}$ such that
$$d(S(t_k)x_k,y_k)<\tfrac{1}{k},\ k\in\N.$$
Consequently, there exists a subsequence $y_{k_j}$ and $y\in\Lambda^{V}(\mathbf{B})$ such that $y_{k_j}\to y$ in $V$,
which shows the compactness of $\Lambda^{V}(\mathbf{B})$ in $V$. Finally, suppose contrary to the claim that 
there exists a bounded subset $D$ of $V$ that is not attracted by $\Lambda^{V}(\mathbf{B})$.
Then there exists $\eps_0>0$, a~sequence $t_k\geq 0$, $t_k\to\infty$ and $x_k\in D$ such that
$$d(S(t_k)x_k,y)>\eps_0,\ k\in\N,\ y\in\Lambda^{V}(\mathbf{B}).$$
Let $t_{D}\geq0$ be such that $S(t_D)D\subseteq \mathbf{B}$. Then $S(t_k)x_{k}=S(t_k-t_{D})S(t_{D})x_k$ for large $k$ and $S(t_k)x_{k}$ has a 
subsequence converging to some $y_0\in\Lambda^{V}(\mathbf{B})$, which is a contradiction.

It remains to show that $\Lambda^{V}(\mathbf{B})$ is invariant. Let $t\geq 0$ and $x\in\Lambda^{V}(\mathbf{B})$. Then there exist $t_{k}\geq0$, $t_{k}\to\infty$ and $x_{k}\in \mathbf{B}$ such that $S(t_k)x_k\to x$.
By assumption there exists a~subsequence $k_{j}$ and $y\in\Lambda^{V}(\mathbf{B})$ such that
$$S(t+t_{k_j})x_{k_j}=S(t)S(t_{k_j})x_{k_{j}}\to y.$$
By the asymptotic closedness of the semigroup it follows that $S(t)x=y$ which proves that $S(t)\Lambda^{V}(\mathbf{B})\subseteq\Lambda^{V}(\mathbf{B})$.
To show the reverse inclusion let  $t\geq 0$ and $y\in\Lambda^{V}(\mathbf{B})$. Then there exist  $t_k\geq t$, $t_{k}\to\infty$
and $x_{k}\in \mathbf{B}$ such that
$S(t_k)x_{k}=S(t)S(t_k-t)x_{k}\to y.$
By assumption there exist $x\in\Lambda^{V}(\mathbf{B})$ and a subsequence $k_{j}$ such that $S(t_{k_j}-t)x_{k_j}\to x.$
Thus the asymptotic closedness implies that $S(t)x=y$ and hence, $\Lambda^{V}(\mathbf{B})\subseteq S(t)\Lambda^{V}(\mathbf{B})$.
We conclude that $\mathbf{A}=\Lambda^{V}(\mathbf{B})$ is the global attractor for the semigroup.

\textbf{Step 2:} $(1)$ implies $(2)$.
Note that any $\eps$-neighborhood $\mathbf{B}$ of $\mathbf{A}$ with $\eps>0$,
$$\mathbf{B}=\bigcup_{x\in\mathbf{A}}B^{V}(x,\eps),$$
is a nonempty bounded absorbing set. Take sequences $t_k\geq0$, $t_k\to\infty$ and $x_{k}\in \mathbf{B}$.
Since $\mathbf{A}$ attracts $\mathbf{B}$, we have
$$\dist^{V}(S(t_k)x_k,\mathbf{A})\leq\dist^{V}(S(t_k)\mathbf{B},\mathbf{A})\to 0\ \text{ as }\ k\to\infty,$$
and by the compactness of $\mathbf{A}$ there exists a convergent subsequence of $S(t_k)x_k$.

\textbf{Step 3:} $(1)$ implies $(3)$.
This is obvious since the global attractor is a compact set attracting all bounded subsets.

\textbf{Step 4:} $(3)$ implies $(1)$.
Note that any $\eps$-neighborhood $\mathbf{B}$ of $K$ with $\eps>0$ is a~bounded absorbing set
and $\dist^{V}(S(t)\mathbf{B},K)\to 0$ as $t\to\infty$. Thus for any sequences $t_{k}\geq 0$, $t_{k}\to\infty$ and $x_{k}\in \mathbf{B}$,
the sequence $S(t_{k})x_{k}$ has a~convergent subsequence since $K$ is compact. Thus, (2) holds and by the equivalence of (1) and (2) there exists a global attractor $\mathbf{A}=\Lambda^{V}(\mathbf{B})$.
By the invariance of the global attractor, $\mathbf{A}$ is contained in $K$, and consequently, $\mathbf{A}=\Lambda^{V}(\mathbf{A})\subseteq\Lambda^{V}(K)$. Conversely, if $x\in\Lambda^{V}(K)$
then there exist sequences $x_k\in K$, $t_k\geq0$, $t_{k}\to\infty$ such that $S(t_k-t_{K})S(t_{K})x_k\to x$, where $t_K\geq 0$ is such that
$S(t_{K})K\subseteq \mathbf{B}$. It follows that $x\in\Lambda^{V}(\mathbf{B})$, which shows that $\mathbf{A}=\Lambda^{V}(K)$.
\end{proof}

\begin{rem}
If a global attractor exists, it is unique which is an immediate consequence of its definition. Moreover, the global attractor is the minimal compact set that attracts all bounded sets. Similar characterizations for the existence of global attractors as in Theorem~\ref{thm:existglobal} and additional equivalent statements were established in \cite{ChCoPa}. However, different notions were used there for the asymptotic closedness of the semigroup as well as for global attractors. The definition of a global attractor in \cite{ChCoPa} was based on the minimality property. 
\end{rem}

We will use Theorem \ref{thm:existglobal} to prove our existence criterion for $T$-discrete exponential attractors. 
Here and in the sequel, given a subset $G$ of a metric space $(V,d)$ and $\eps>0$ we denote by $N^{V}(G,\eps)$ the minimal number of open $\eps$-balls in $V$ centered at points from $G$ necessary to cover the set $G$. 

\begin{defn}\label{defn:EXPAT}
An \emph{exponential attractor} for a semigroup $\{S(t)\colon t\geq 0\}$ on a metric space $(V,d)$ is a~nonempty compact set $\mathbf{M}\subseteq V$ such that
\begin{itemize}
\item[(i)] $\mathbf{M}$ is positively invariant under the semigroup, i.e., $S(t)\mathbf{M}\subseteq\mathbf{M}$ for all $t\geq 0,$
\item[(ii)] the fractal dimension of $\mathbf{M}$ in $V$ is finite with a given bound $\chi\geq 0$,
i.e.,
$$\dimf^{V}(\mathbf{M})=\limsup_{\eps\to0^{+}}\log_{\frac{1}{\eps}}N^{V}(\mathbf{M},\eps)\leq\chi<\infty,$$
\item[(iii)]  $\mathbf{M}$ is exponentially attracting, i.e., there exists $\xi>0$ such that for every bounded subset $G$ of $V$ we have
$$\lim_{t\to\infty}e^{\xi t}\dist^{V}(S(t)G,\mathbf{M})=0.$$
\end{itemize}

If we replace the positive invariance (i) of $\mathbf{M}$ by the weaker requirement of positive $T$-invariance, i.e.,
\begin{itemize}
\item[(i')] there exists  $T>0$ such that $S(T)\mathbf{M}\subseteq\mathbf{M},$
\end{itemize}
then we call $\mathbf{M}$ a \emph{$T$-discrete exponential attractor} for the semigroup $\{S(t)\colon t\geq 0\}$.
\end{defn}

Throughout the paper, we denote $T$-discrete exponential attractors by $\mathbf{M_0}$ and exponential attractors in the classical sense by $\mathbf{M}$.

For time discrete semigroups, i.e., when $t\in\mathbb{N}_0$,  an exponential attractor in the classical sense exists if and only if a $T$-discrete exponential attractor exists, cf.~Theorem~\ref{thm:SMALLEXPATTR}.
For semigroups with time $t\in[0,\infty)$ $T$-discrete exponential attractors satisfy all properties of an exponential attractor except for the positive invariance. In fact, they are only positively invariant with respect to discrete times $kT>0, k\in\mathbb{N}$. However, the time step $T$ can be chosen arbitrarily small, as shown in Theorem \ref{thm:SMALLEXPATTR}.  
Moreover, if a $T$-discrete exponential attractor $\mathbf{M_0}$ exists, the global attractor exists and is contained in $\mathbf{M_0}$. Of course, this latter statement also applies to exponential attractors in the classical sense.

Our criterion for the existence of a $T$-discrete exponential attractor for a semigroup on a~complete metric space is based on a covering condition, similarly as in \cite[Theorem 2.1]{Pra03b}. This theorem was already announced in \cite{CzKa}.

\begin{thm}\label{thm:EXDEXPAT}
Let $\{S(t)\colon t\geq 0\}$ be an asymptotically closed semigroup on a complete metric space $(V,d)$ and let $T>0$.
Then, the following statements are equivalent:
\begin{itemize}
\item[$(1)$] There exists a $T$-discrete exponential attractor $\mathbf{M_0}$ in $V$ for the semigroup.
\item[$(2)$] There exists a nonempty bounded 
absorbing set $\mathbf{B}\subseteq V$ for the semigroup
such that the covering condition 
\begin{equation}\label{e:CONDBALLSSEMB}
N^{V}(S(kT)\mathbf{B},aq^{k})\leq b h^{k},\ k\in\N,\ k\geq k_0,
\end{equation}
holds for some $k_0\in\N$, $a,b>0$, $q\in(0,1)$ and $h\geq 1$. 
\end{itemize}
Moreover, if the covering condition \eqref{e:CONDBALLSSEMB} holds, then
$$\mathbf{M_0}=\mathbf{A}\cup\mathbf{E_0}=\cl_{V}\mathbf{E_0}\subseteq \mathbf{B},$$
is a $T$-discrete exponential attractor with rate of attraction
$\xi\in(0,\frac{1}{T}\ln{\frac{1}{q}})$, and its fractal dimension is bounded by
\begin{equation}\label{e:ESTDIMM3}
\dimf^{V}(\mathbf{M_0})\leq\log_{\frac{1}{q}}h.
\end{equation}
Here, $\mathbf{E_0}$ is a certain countable subset of $\mathbf{B}$ and $\mathbf{A}=\Lambda^{V}(\cl_{V}\mathbf{E_0})$ is the global attractor for the semigroup.
\end{thm}

\begin{rem}\label{rem:POSINV}
Without loss of generality, in Theorem~\ref{thm:EXDEXPAT} we can assume that the absorbing set $\mathbf{B}$ is \emph{positively invariant}.
Indeed, if $B_0$ is a~bounded absorbing set that satisfies the covering condition \eqref{e:CONDBALLSSEMB} 
and $S(t)B_0\subseteq B_0$ for all $t\geq t_{B_0}$, $t_{B_0}\geq0$ being the absorbing time for $B_0$, then the positively invariant subset
\begin{equation*}
\mathbf{B}=\bigcup_{t\geq t_{B_0}}S(t)B_0\subseteq B_0
\end{equation*}
is a bounded absorbing set that satisfies \eqref{e:CONDBALLSSEMB} with a possibly  larger constant $a>0$.
\end{rem}

\begin{proof}[Proof of Theorem~\ref{thm:EXDEXPAT}]
\textbf{Step 1.}
We show that the existence of a~$T$-discrete exponential attractor 
$\mathbf{M_0}$ implies condition (2).  
Note that given $\eps_0>0$, the $\eps_0$-neighborhood $\mathbf{B}$ of $\mathbf{M_0}$,
$$\mathbf{B}=\bigcup_{x\in\mathbf{M_0}}B^{V}(x,\eps_0),$$
is a nonempty bounded absorbing set, since $\mathbf{M_0}$ attracts every bounded subset of $V$.
Moreover, due to the exponential rate of attraction for some $\xi>0$
there exists $s_{\mathbf{B}}\geq 0$ such that
\begin{equation*}
S(t)\mathbf{B}\subseteq\bigcup_{x\in\mathbf{M_0}}B^{V}(x,e^{-\xi t}),\ t\geq s_{\mathbf{B}}.
\end{equation*}
Since $\dimf^{V}(\mathbf{M_0})<\chi_0$ for some $\chi_0>0$, it follows that 
$N^{V}(\mathbf{M_0},\eps)<e^{-\chi_0\ln\eps}$ for all sufficiently small $\eps>0$.
Thus, setting $q=e^{-\xi T}\in(0,1)$, we find $k_0\in\N$ such that for $k\geq k_0$ we have 
$$N^{V}(\mathbf{M_0},q^{k})< e^{\chi_{0}\xi Tk}\quad \text{and}\quad S(kT)\mathbf{B}\subseteq\bigcup_{x\in\mathbf{M_0}}B^{V}(x,q^{k}).$$
Consequently, we obtain
$$N^{V}(S(kT)\mathbf{B},4q^{k})\leq N^{V}(\mathbf{M_0},q^{k})< e^{\chi_{0}\xi Tk},\ k\geq k_0,$$
which shows that the covering condition \eqref{e:CONDBALLSSEMB} holds with  $h=e^{\chi_{0}\xi T}$,  $a=4$ and $b=1$.

\textbf{Step 2.} 
We show the reverse statement in several steps. Assume that (2) holds with $\mathbf{B}$ being positively invariant which we can assume  due to Remark~\ref{rem:POSINV}.  We first prove that there exists a~countable subset $\mathbf{E_0}$ of $\mathbf{B}$ that 
is precompact in $V$ and such that $S(T)\mathbf{E_0}\subseteq \mathbf{E_0}$,
\begin{itemize}
\item[($e_1$)] $\displaystyle\mathbf{E_0}=\bigcup_{k\geq k_0}Q_{k}$,
where $Q_{k}\subseteq S(kT)\mathbf{B}$ is finite with $\displaystyle\# Q_{k}\leq b\sum_{l=0}^{k-k_0}h^{k-l}$, and
\begin{equation*}
\dimf^{V}(\mathbf{E_0})\leq\log_{\frac{1}{q}}h,
\end{equation*}
\item[($e_2$)] for any $\xi\in(0,\xi_{T})$, where $\xi_T=\frac{1}{T}\ln\frac{1}{q}>0$,
and any bounded subset $G$ of $V$ we have
\begin{equation*}
\lim_{t\rightarrow\infty}e^{\xi t}\dist^{V}(S(t)G,\mathbf{E_0})=0.
\end{equation*}
\end{itemize}
To this end, let $W_{k}$, $k\geq k_0$, be the centers of the balls from the coverings in \eqref{e:CONDBALLSSEMB}, so that
\begin{equation}\label{e:sets_w}
W_{k}\subseteq S(kT)\mathbf{B}\subseteq\mathbf{B},\quad 
\# W_{k}\leq b\displaystyle h^{k}, \quad 
\displaystyle S(kT)\mathbf{B}\subseteq\bigcup_{x\in W_{k}}B^{V}(x, a q^k).
\end{equation}
We now set $Q_{k_0}=W_{k_0}$ and define the sets $Q_k$ recursively by
$$Q_{k}=W_{k}\cup S(T)Q_{k-1},\ k>k_0.$$
Then using \eqref{e:sets_w} it follows that for $k\geq k_0$ the sets $Q_{k}$ satisfy
\begin{itemize}
\item[($q_1$)] $S(T)Q_{k}\subseteq Q_{k+1}$, $\quad Q_{k}\subseteq S(kT)\mathbf{B}\subseteq\mathbf{B}$,
\item[($q_2$)] $\displaystyle Q_{k}=\bigcup_{l=0}^{k-k_0}S(lT)W_{k-l}$,
$\quad\displaystyle\# Q_{k}\leq b\sum_{l=0}^{k-k_0}h^{k-l}$.
\end{itemize}
Indeed, the first statement in ($q_1$) follows from the definition of $Q_{k}$
and the second one by induction and \eqref{e:sets_w}, since
$$Q_{k+1}=W_{k+1}\cup S(T)Q_{k}\subseteq S((k+1)T)\mathbf{B}\cup S(T)S(kT)\mathbf{B}=S((k+1)T)\mathbf{B}\subseteq \mathbf{B}.$$
The first statement in ($q_2$) follows by induction, since
\begin{equation*}
\begin{split}
Q_{k+1}&=W_{k+1}\cup S(T)Q_{k}=W_{k+1}\cup\bigcup_{l=0}^{k-k_0}S((l+1)T)W_{k-l}\\
&=W_{k+1}\cup\bigcup_{m=1}^{k+1-k_0}S(mT)W_{k+1-m}=\bigcup_{l=0}^{k+1-k_0}S(lT)W_{k+1-l},
\end{split}
\end{equation*}
and the second statement in ($q_2$) then follows from \eqref{e:sets_w}.

We now define
$$\mathbf{E_0}=\bigcup_{k\geq k_0}Q_{k}$$
and observe that
$$\mathbf{E_0}=\bigcup_{k=k_0}^{\infty}\bigcup_{l=0}^{k-k_0}S(lT)W_{k-l}
=\bigcup_{l=0}^{\infty}\bigcup_{m=k_0}^{\infty}S(lT)W_{m}.$$
The set $\mathbf{E_0}$ is a~nonempty subset of $\mathbf{B}$ and by ($q_1$) we have
$$S(T)\mathbf{E_0}=\bigcup_{k\geq k_0}S(T)Q_{k}\subseteq\bigcup_{k\geq k_0}Q_{k+1}\subseteq\mathbf{E_0}.$$
Moreover,  ($q_1$) implies that for any $l\geq k\geq k_0$ we have
$$Q_{l}\subseteq S(lT)\mathbf{B}=S(kT)S((l-k)T)\mathbf{B}\subseteq S(kT)\mathbf{B}.$$
Consequently, for all $k\geq k_0$ we obtain
$$\mathbf{E_0}=\bigcup_{l=k_0}^{k}Q_{l}\cup\bigcup_{l=k+1}^{\infty}Q_{l}
\subseteq\bigcup_{l=k_0}^{k}Q_{l}\cup S(kT)\mathbf{B},$$
and using \eqref{e:sets_w} and ($q_2$) we conclude that for $k\geq k_0$
\begin{align}\label{eq:proofE0}
N^{V}(\mathbf{E_0},aq^{k})\leq\#\left(\bigcup_{l=k_0}^{k}Q_{l}\right)+\# W_{k}\leq
b\sum_{l=k_0}^{k}\sum_{m=0}^{l-k_0}h^{l-m}+bh^{k}\leq 2b(k-k_0+1)^{2}h^{k}.
\end{align}
Consider any sequence $\eps_{n}>0$, $n\in\N$, converging to $0$ and choose integers $k_n\in\N$ such that
$$k_n\geq k_0\ \text{ and }\ a q^{k_{n}}\leq\eps_n<a q^{k_n-1}<1\ \text{ for large}\ n.$$
Since $N^{V}(\mathbf{E_0},\eps_n)\leq N^{V}(\mathbf{E_0},a q^{k_{n}})$ and $k_n\to\infty$, it follows from \eqref{eq:proofE0} that
$$\log_{\frac{1}{\eps_n}}N^{V}(\mathbf{E_0},\eps_n)\leq\frac{\ln(2b)+2\ln(k_n)+k_{n}\ln{h}}{-\ln{a}-(k_{n}-1)\ln{q}},$$
which shows that  $\mathbf{E_0}$ is precompact in $V$ and that ($e_1$) holds.

By \eqref{e:sets_w} we have for $k\geq k_0$
$$\dist^{V}(S(kT)\mathbf{B},\mathbf{E_0})\leq\dist^{V}(S(kT)\mathbf{B},W_{k})\leq a q^{k}.$$
Moreover, for a fixed $0<\xi<\xi_T=\frac{1}{T}\ln{\frac{1}{q}}$ we have
$$e^{\xi Tk}a q^{k}=a e^{(\xi T+\ln{q})k}\to 0\ \text{ as }k\to\infty,$$
which yields
$$e^{\xi Tk}\dist^{V}(S(kT)\mathbf{B},\mathbf{E_0})\to 0\ \text{ as }k\to\infty.$$
For fixed $0<\xi<\xi_T$ and  $\eps>0$ let $k_\eps=k_{\eps}(\xi,\eps)\in\N$ be such that
$$e^{\xi T}e^{\xi T k}\dist^{V}(S(kT)\mathbf{B},\mathbf{E_0})<\eps\ \text{ for}\ k\geq k_\eps.$$
Set $t_\eps=k_\eps T$ and let $t\geq t_\eps$. Then $t=kT+t_0$ for some $k\geq k_\eps$, $k\in\N$,
and $t_0\in[0,T)$, and by the positive invariance of $\mathbf{B}$ we conclude that 
$$e^{\xi t}\dist^{V}(S(t)\mathbf{B},\mathbf{E_0})=e^{\xi t}\dist^{V}(S(kT)S(t_0)\mathbf{B},\mathbf{E_0})
\leq e^{\xi T}e^{\xi T k}\dist^{V}(S(kT)\mathbf{B},\mathbf{E_0})<\eps.$$
It remains to show that the set $\mathbf{E_0}$ is exponentially attracting.
Let $G\subseteq V$ be bounded and $t_{G}\geq 0$ such that
$S(t_G)G\subseteq \mathbf{B}$. We fix $0<\xi<\xi_T$ and $\eps>0$ and find as above $t_\eps\geq 0$ such that
$$e^{\xi t_{G}}e^{\xi t}\dist^{V}(S(t)\mathbf{B},\mathbf{E_0})<\eps,\ t\geq t_\eps.$$
Then for $t\geq t_{G}+t_\eps$ we have
\begin{align*}
e^{\xi t}\dist^{V}(S(t)G,\mathbf{E_0})&=e^{\xi t_{G}}e^{\xi (t-t_{G})}\dist^{V}(S(t-t_{G})S(t_{G})G,\mathbf{E_0})\\
&\leq e^{\xi t_{G}}e^{\xi(t-t_{G})}\dist^{V}(S(t-t_{G})\mathbf{B},\mathbf{E_0})<\eps,
\end{align*}
which shows ($e_2$).

\textbf{Step 3.}
We now define the $T$-discrete exponential attractor as $\mathbf{M_0}=\cl_{V}\mathbf{E_0}$.
Note that $\mathbf{M_0}$ is nonempty and compact, since the space $V$ is complete, and $\mathbf{M_0}$ attracts all bounded subsets of $V$ at an exponential rate $\xi\in(0,\frac{1}{T}\ln \frac{1}{q})$ by  ($e_2$).
Hence, Theorem \ref{thm:existglobal} implies that the global attractor exists, $\mathbf{A}=\Lambda^{V}(\mathbf{M_0})$ and by the minimality of the global attractor, $\mathbf{A}\subseteq \mathbf{M_0}$.  Moreover, $\mathbf{A}\subseteq\mathbf{B}$ as 
$\mathbf{B}$ is an absorbing set, and since $\mathbf{A}$ is invariant, it follows that
$$\mathbf{A}=S(kT)\mathbf{A}\subseteq S(kT)\mathbf{B},\ k\in\N.$$
Together with \eqref{e:CONDBALLSSEMB} this implies that
$$N^{V}(\mathbf{A}, 2a q^k)\leq bh^k,\ k\geq k_0,$$
and consequently, 
\begin{equation}\label{e:ESTGLOBSEM}
\dimf^{V}(\mathbf{A})\leq\log_{\frac{1}{q}}h.
\end{equation}
It remains to show that 
\begin{align}\label{eq:proofEAstr}
\mathbf{M_0}=\mathbf{A}\cup\mathbf{E_0}.
\end{align}
Indeed, then  $\mathbf{M_0}\subseteq \mathbf{B}$ since $\mathbf{E_0}\subseteq \mathbf{B}$ and $\mathbf{A}\subseteq \mathbf{B}$. 
Moreover, we have
$$S(kT)\mathbf{M_0}=S(kT)\mathbf{A}\cup S(kT)\mathbf{E_0}\subseteq \mathbf{A}\cup \mathbf{E_0}=\mathbf{M_0},\ k\in\N,$$
and  by ($e_1$) and \eqref{e:ESTGLOBSEM} the fractal dimension of $\mathbf{M_0}$ is bounded by 
\begin{equation*}
\dimf^{V}(\mathbf{M_0})=\max\{\dimf^{V}(\mathbf{A}),\dimf^{V}(\mathbf{E_0})\}\leq\log_{\frac{1}{q}}h.
\end{equation*}

To prove \eqref{eq:proofEAstr} we first observe that $\mathbf{A}\cup\mathbf{E_0}\subseteq \cl_{V}\mathbf{E_0}=\mathbf{M_0}$, as $\mathbf{A}\subseteq\mathbf{M_0}$. To show the reverse inclusion
let $x\in\cl_{V}\mathbf{E_0}$. Then there exists a sequence $x_l\in\mathbf{E_0}$, $l\in\N$, such that $x_{l}\to x$. 
Moreover, for every $l\in\N$ there exists $k_l\in\N$ such that $k_{l}\geq k_0$ and $x_l\in Q_{k_l}$.
If $p=\sup\{k_l\colon l\in\N\}<\infty$, the sequence $x_l$ is contained in the finite set
$\bigcup_{k=k_0}^{p}Q_k$ and consequently, $x\in\bigcup_{k=k_0}^{p}Q_k\subseteq\mathbf{E_0}$.
Otherwise, if $\sup\{k_l\colon l\in\N\}=\infty$, there exists a~subsequence $k_{l_j}$ such that $k_{l_{j}}\to\infty$.
Since $x_{l_j}\in Q_{k_{l_{j}}}\subseteq S(k_{l_j}T)\mathbf{B}$, the global attractor $\mathbf{A}$ attracts $\mathbf{B}$ and $\mathbf{A}$ is compact, we conclude that $x\in\mathbf{A}$. It follows that $x\in\mathbf{A}\cup\mathbf{E_0}$.
\end{proof}

\begin{rem}\label{rem:closed}
Assume that the condition (2) in Theorem~\ref{thm:EXDEXPAT} holds and $S(T)$ is a closed map on $\cl_{V}\mathbf{B}$, that is,
for any sequence $x_k\in\cl_{V}\mathbf{B}$ the following implication holds:
$$\text{if }x_k\to x\text{ and }S(T)x_k\to y\text{ with }x,y\in\cl_{V}\mathbf{B},\text{ then }S(T)x=y.$$   
Then the asymptotic closedness of the semigroup is not required to prove statement (1). Moreover, the $T$-discrete exponential attractor $\mathbf{M_0}=\cl_{V}\mathbf{E_0}$ 
is a subset of $\cl_{V}\mathbf{B}$, with rate of attraction $\xi\in(0,\frac{1}{T}\ln\frac{1}{q})$, and
its fractal dimension is bounded as in~\eqref{e:ESTDIMM3}, where $\mathbf{E_0}$ is a~certain countable subset of $\cl_{V}\mathbf{B}$.
Indeed, this follows from Step 2 in the proof of Theorem~\ref{thm:EXDEXPAT} and the $T$-positive invariance of  $\mathbf{M_0}=\cl_{V}\mathbf{E_0}$ which holds due to the closedness of $S(T)$ on $\cl_{V}\mathbf{B}$. 
\end{rem}

Note that the asymptotic closedness of the semigroup is assumed in Theorem~\ref{thm:existglobal} to conclude the existence of a global attractor. If we know in advance that the semigroup possesses a global attractor, neither the asymptotic closedness of the semigroup nor the completeness of the metric space is required to prove statement (1) in Theorem~\ref{thm:EXDEXPAT}.

\begin{cor}\label{cor:WITHGLOBALATTR}
Let $\{S(t)\colon t\geq 0\}$ be a semigroup on a metric space $(V,d)$, which possesses a~global attractor $\mathbf{A}$ in $V$.
If the covering condition \eqref{e:CONDBALLSSEMB} holds for a nonempty bounded 
absorbing set $\mathbf{B}\subseteq V$, then there exists a $T$-discrete exponential attractor $\mathbf{M_0}=\mathbf{A}\cup\mathbf{E_0}=\cl_{V}\mathbf{E_0}\subseteq\mathbf{B}$ in $V$, with rate of attraction $\xi\in(0,\frac{1}{T}\ln\frac{1}{q})$, and its fractal dimension is bounded as in \eqref{e:ESTDIMM3}, where $\mathbf{E_0}$ is some countable subset of $\mathbf{B}$.
\end{cor}

\begin{proof}
Having constructed $\mathbf{E_0}$ as in Step 2 in the proof of Theorem~\ref{thm:EXDEXPAT}, we define $\mathbf{M_0}=\mathbf{A}\cup\mathbf{E_0}$.
We easily see that $S(T)\mathbf{M_0}\subseteq\mathbf{M_0}$, $\mathbf{M_0}$ attracts all bounded subsets of $V$ at an exponential rate $\xi\in(0,\frac{1}{T}\ln\frac{1}{q})$ and its fractal dimension is bounded as in \eqref{e:ESTDIMM3}. To justify its compactness in $V$, we note that a sequence $x_l\in\mathbf{M_0}$ either contains a~subsequence in $\mathbf{A}$, which in turn has a~subsequence converging to an element of $\mathbf{A}\subseteq\mathbf{M_0}$, or it contains a~subsequence in $\mathbf{E_0}$. In the latter case, either it is contained in a finite set and hence has a~convergent subsequence to an element of $\mathbf{E_0}$, or it possesses a~subsequence $x_{l_j}\in S(k_{l_j}T)\mathbf{B}$, which is attracted by the global attractor $\mathbf{A}$. Thus it has a convergent subsequence to some $x\in\mathbf{A}\subseteq\mathbf{M_0}$.
The claim that $\mathbf{A}\cup\mathbf{E_0}=\cl_{V}\mathbf{E_0}$ follows the lines of Step 3 in the proof of Theorem~\ref{thm:EXDEXPAT}. 
\end{proof}

Next we prove that the existence of a $T$-discrete exponential attractor implies the existence of a $\widetilde T$-discrete exponential attractor for  \emph{arbitrarily small} $\widetilde T>0$, i.e., the positive invariance holds with respect to arbitrarily small time steps, see \cite{sonner2012}. On the other hand, we note that $\mathbf{M_0}$ is a $kT$-discrete exponential attractor for any $k\in\mathbb{N}$ if $\mathbf{M_0}$ is a $T$-discrete exponential attractor. Hence, there also exists a $\widetilde T$-discrete exponential attractor for arbitrarily large $\widetilde T$.

\begin{thm}\label{thm:SMALLEXPATTR}
Let $\{S(t)\colon t\geq 0\}$ be an asymptotically closed semigroup on a complete metric space $(V,d)$, $T>0$ 
and $N\in\N$.
Then, the following statements are equivalent:
\begin{itemize}
\item[$(1)$] There exists a $T$-discrete exponential attractor $\mathbf{M_0}$ in $V$ for the semigroup.
\item[$(2)$] There exists a $\frac{T}{N}$-discrete exponential attractor $\mathbf{\widetilde{M}_0}$ in $V$ for the semigroup. 
\end{itemize}
\end{thm}

\begin{proof}
Assume that $\mathbf{M_0}$ is a $T$-discrete exponential attractor. By Theorem~\ref{thm:EXDEXPAT} there exists a~positively invariant bounded absorbing set $\mathbf{B}\subseteq V$ for the semigroup such that \eqref{e:CONDBALLSSEMB} holds, that is,
\begin{equation*}
N^{V}(S(kT)\mathbf{B},aq^{k})\leq b h^{k},\ k\in\N,\ k\geq k_0,
\end{equation*}
for some $k_0\in\N$, $a,b>0$, $q\in(0,1)$ and $h\geq 1$. Moreover, we have
$$\mathbf{M_0}=\mathbf{A}\cup\mathbf{E_0}=\cl_{V}\mathbf{E_0}\subseteq \mathbf{B},$$
where $\mathbf{A}$ is the global attractor for the semigroup and $\displaystyle\mathbf{E_0}=\bigcup_{k\geq k_0}Q_k$
with $Q_k$ being finite subsets of $S(kT)\mathbf{B}$ with $\displaystyle\# Q_{k}\leq b\sum_{l=0}^{k-k_0}h^{k-l}$.

We define 
$$\widetilde{Q}_{k}:=\bigcup_{l=0}^{N-1}S\left(\tfrac{lT}{N}\right)Q_k,\quad \mathbf{\widetilde{E}_0}:=\bigcup_{l=0}^{N-1}S\left(\tfrac{lT}{N}\right)\mathbf{E_0}=\bigcup_{k\geq k_0}\widetilde{Q}_{k},\quad \mathbf{\widetilde{M}_0}:=\cl_{V}\mathbf{\widetilde{E}_0}.$$
Each $\widetilde{Q}_{k}$ is a finite subset of 
$$\bigcup_{l=0}^{N-1}S\left(\tfrac{lT}{N}\right)S(kT)\mathbf{B}\subseteq S(kT)\mathbf{B}\subseteq\mathbf{B}\ \text{ and }\ \#\widetilde{Q}_{k}\leq bN\sum_{l=0}^{k-k_0}h^{k-l}.$$
For $l\geq k\geq k_0$ we have
$\widetilde{Q}_{l}\subseteq S(kT)S((l-k)T)\mathbf{B}\subseteq S(kT)\mathbf{B},$
and hence,
$$\mathbf{\widetilde{E}_0}\subseteq \bigcup_{l=k_0}^{k}\widetilde{Q}_{l}\cup S(kT)\mathbf{B}.$$
By \eqref{e:CONDBALLSSEMB} we know that
$$N^{V}(S(kT)\mathbf{B}\cap\mathbf{\widetilde{E}_0},2aq^{k})\leq bh^{k},\ k\geq k_0$$
and consequently we conclude that
$$N^{V}(\mathbf{\widetilde{E}_0},2aq^{k})\leq 2bN(k-k_0+1)^{2}h^{k},\ k\geq k_0.$$
Reasoning as in Step~2 of the proof of Theorem~\ref{thm:EXDEXPAT} we conclude that $\mathbf{\widetilde{E}_0}$ is precompact and
\begin{equation}\label{e:ESTTILDEE0}
\dimf^{V}(\mathbf{\widetilde{E}_0})\leq\log_{\frac{1}{q}}h.
\end{equation}
Since $\mathbf{M_0}\subseteq\mathbf{\widetilde{M}_0}$, for any $\xi\in(0,\frac{1}{T}\ln{\frac{1}{q}})$ and any  bounded set $G$ in $V$ we have
$$e^{\xi t}\dist^{V}(S(t)G,\mathbf{\widetilde{M}_0})\to 0\ \text{ as}\ t\to\infty,$$
i.e., $\mathbf{\widetilde{M}_0}$ exponentially attracts all bounded sets. 

The set $\mathbf{\widetilde{M}_0}$ is compact and as in the proof of Theorem \ref{thm:EXDEXPAT} we show that 
$\mathbf{\widetilde{M}_0}=\mathbf{A}\cup\mathbf{\widetilde{E}_0}$, and consequently, 
$\mathbf{\widetilde{M}_0}=\bigcup_{l=0}^{N-1}S\left(\tfrac{lT}{N}\right)\mathbf{M_0}.$
Furthermore, we observe that
$$S(\tfrac{T}{N})\mathbf{\widetilde{M}_0}=\bigcup_{l=1}^{N-1}S(\tfrac{lT}{N})\mathbf{M_0}\cup S(T)\mathbf{M_0}\subseteq\bigcup_{l=0}^{N-1}S(\tfrac{lT}{N})\mathbf{M_0}=\mathbf{\widetilde{M}_0},$$
which shows the positive invariance with respect to the time step $\frac{T}{N}$. 

Note that  $\dimf^{V}(\mathbf{A})\leq\log_{\frac{1}{q}}h$ (see \eqref{e:ESTGLOBSEM}), and thus by \eqref{e:ESTTILDEE0}
we get the estimate
$$\dimf^{V}(\mathbf{\widetilde{M}_0})=\max\{\dimf^{V}(\mathbf{A}),\dimf^{V}(\mathbf{\widetilde{E}_0})\}\leq\log_{\frac{1}{q}}h.$$ 
Hence $\mathbf{\widetilde{M}_0}$ is a $\frac{T}{N}$-discrete exponential attractor for the semigroup.

Conversely, if $\mathbf{\widetilde{M}_0}$ is a $\frac{T}{N}$-discrete exponential attractor, then $\mathbf{M_0}=\mathbf{\widetilde{M}_0}$ is also a $T$-discrete exponential attractor, since applying $N$ times the inclusion $S(\frac{T}{N})\mathbf{M_0}\subseteq\mathbf{M_0}$, we get $S(T)\mathbf{M_0}\subseteq\mathbf{M_0}$. 
\end{proof}

\begin{rem}
Observe that we obtain the same upper bound for the fractal dimension and the same rate of exponential attraction for  
the exponential attractors $\mathbf{M_0}$ and $\mathbf{\widetilde{M}_0}$  in the above proof.
\end{rem}

\section{Construction based on quasi-stability}\label{sec:QS}

There exist different approaches to construct exponential attractors for semigroups.
We compare several broadly used methods and show that the assumptions lead to the covering condition \eqref{e:CONDBALLSSEMB} with specific constants $h$ and $q$ determining the bound for the fractal dimension and exponential rate of attraction of the exponential attractor in Theorem \ref{thm:EXDEXPAT}. The most general method is based on the quasi-stability of a~semigroup, as introduced by I.~Chueshov in \cite[Definition 3.4.1]{chueshov} (see also \cite{cl2008}), that we address in this section. 
  
We first recall that a pseudometric space $(A,\rho)$ is a nonempty set $A$ with a function $\rho\colon A\times A\to[0,\infty)$ that
is symmetric, satisfies the triangle inequality and $\rho(x,x)=0$ for $x\in A$. We say that $(A,\rho)$ is precompact
 if each sequence in $A$ contains a Cauchy subsequence with respect to $\rho$.
Equivalently, this means that $A$ is totally bounded, i.e., for any $\eps>0$ there exists
a finite cover of $A$ by open $\eps$-balls centered at points from $A$.

Given a~precompact pseudometric space $(A,\rho)$ and a nonempty subset $F\subseteq A$, we denote by $m_{\rho}(F,\eps)$  the maximal cardinality of an $\eps$-distinguishable subset $U$ of $F$, i.e.,
\begin{equation*}
\rho(x,y)\geq\eps,\ x,y\in U\subseteq F,\ x\neq y.
\end{equation*}

\begin{rem}
Note that for any $\eps>0$ we have
$$1\leq m_{\rho}(F,\eps)\leq m_{\rho}(A,\eps)<\infty.$$
Indeed, by assumption there are no $\eps$-distinguishable subsets of $F$
which contain a countably infinite number of points, that is, all these sets are finite.
If we consider a family ${\mathcal F}^{\eps}$ of $\eps$-distinguishable subsets of $F$ with the inclusion relation, then
each chain of subsets will have an upper bound given by the union of sets in this chain. Consequently,
each set in this family is contained in a certain maximal element of ${\mathcal F}^{\eps}$.

Suppose that the cardinalities of these maximal elements of ${\mathcal F}^{\eps}$ are unbounded. Then, choosing a~maximal element
$\hat {\mathcal X}$ in ${\mathcal F}^{\frac{\eps}{4}}$ with $\hat {\mathcal X}$ consisting of points denoted by
$\hat x_{k}\in F$, $k=1,\ldots, n_{\hat {\mathcal X}}$, we find a maximal element of ${\mathcal F}^{\eps}$ which contains at least
$n_{\hat {\mathcal X}}+1$ points outside of $\hat {\mathcal X}$. These points lie in the union of $n_{\hat {\mathcal X}}$ balls
$\{x\in F\colon\rho(x,\hat x_{k})<\frac{\eps}{4}\}$. Thus there are at least two points $y,z\in F$ satisfying
\begin{equation}\label{e:EPS0}
\rho(y,z)\geq \eps
\end{equation}
which also satisfy for some $\hat x_{k}$
$$\rho(y,\hat x_{k})<\frac{\eps}{4}\ \text{ and } \ \rho(z, \hat x_{k})< \frac{\eps}{4}.$$
On the other hand, we have
$$\rho(y, z) \leq \rho(y, \hat x_{k}) + \rho(z, \hat x_{k})<\frac{\eps}{2},$$
which contradicts \eqref{e:EPS0}. Thus, the cardinalities of maximal elements of ${\mathcal F}^{\eps}$ remain bounded.
\end{rem}

To show that the quasi-stability of a semigroup implies the covering condition \eqref{e:CONDBALLSSEMB}
we use the following fundamental lemma from \cite[p.~25]{cl2008}.
In the sequel, $\widehat{N}^{V}(B,\eps)$ denotes the minimal number of subsets of $B$ in a metric space $(V,d)$ with diameter no larger than $2\eps$ necessary to cover the set $B$. 

\begin{lem}(cf. \cite[p.~25]{cl2008}) \label{lem:FUNDAMENTAL}
Let $A$ be a nonempty subset of a metric space $(V,d)$ and assume that there is
a~pseudometric $\rho$ on $A$ such that $(A,\rho)$ is a precompact pseudometric space.
Suppose that for a~map $S\colon A\to V$ there exists $\eta\geq 0$ such that
\begin{equation}\label{e:MAININEQ}
d(S(x),S(y))\leq\eta d(x,y)+\rho(x,y),\ x,y\in A.
\end{equation}
If $\widehat{N}^{V}(A,\eps)<\infty$ 
for some $\eps>0$, then for any $\sigma>0$ we have
\begin{equation}\label{e:CRUCIAL}
\widehat{N}^{V}(S(A),(\eta+\sigma)\eps)\leq \widehat{N}^{V}(A,\eps)c_\rho(A,\eps,\sigma\eps),
\end{equation}
where
\begin{equation}\label{e:DEFCRHO}
c_\rho(A,\eps,\mu):=\sup\left\{m_{\rho}(F,\mu)\colon\emptyset\neq F\subseteq A,\ \diam^{V}(F)\leq 2\eps\right\}\leq {m_{\rho}(A,\mu)}.
\end{equation}
\end{lem}

\begin{proof}
By assumption 
$\widehat{N}=\widehat{N}^{V}(A,\eps)<\infty$, and thus, we have
$$A=\bigcup_{i=1}^{\widehat{N}}F_i,$$
where $\emptyset\neq F_i\subseteq A$ and $\diam^{V}(F_i)\leq 2\eps$.
We fix $\sigma>0$ and set $m_i=m_{\rho}(F_i,\sigma\eps)\in\N$.
Let $\{x_1^{i},\ldots,x_{m_i}^{i}\}\subseteq F_i$ be a~$\sigma\eps$-distinguishable subset of $F_i$ in $(A,\rho)$. Then,  we have
$$\rho(x_j^{i},x_{l}^{i})\geq\sigma\eps,\ j\neq l,$$
and
$$m_i=m_{\rho}(F_i,\sigma\eps)\leq c_\rho(A,\eps,\sigma\eps).$$
It follows that
$$F_i=\bigcup_{j=1}^{m_i}C_j^{i},\quad C_j^{i}=\{x\in F_i\colon\rho(x,x_j^{i})<\sigma\eps\}.$$
Indeed, let $x\in F_i$ and note that if $x=x_j^{i}$ for some $j\in\{1,\ldots,m_i\}$, then
$x\in C_j^{i}$. On the other hand, if $x\neq x_j^{i}$ for any $j\in\{1,\ldots, m_i\}$, then from the maximality of $m_i$
it follows that $\rho(x,x_{j_0}^{i})<\sigma\eps$ for some $j_0\in\{1,\ldots, m_i\}$ and hence, $x\in C_{j_0}^{i}$.

Consequently, we obtain
$$A=\bigcup_{i=1}^{\widehat{N}}\bigcup_{j=1}^{m_i}C_j^{i}\ \mbox{ and }\ S(A)=\bigcup_{i=1}^{\widehat{N}}\bigcup_{j=1}^{m_i}S(C_j^{i}).$$
Note that if $x,y\in C_j^{i}\subseteq F_i$ then $\diam^{V}(C_j^{i})\leq\diam^{V}(F_i)\leq 2\eps$
and
$$\rho(x,y)\leq\rho(x,x_j^{i})+\rho(x_j^{i},y)<2\sigma\eps.$$
Applying \eqref{e:MAININEQ}, we have
$$d(S(x),S(y))<\eta d(x,y)+2\sigma\eps,$$
which yields
$$\diam^{V}(S(C_j^{i}))\leq\eta\diam^{V}(C_j^{i})+2\sigma\eps\leq2(\eta+\sigma)\eps$$
and in consequence \eqref{e:CRUCIAL}.
\end{proof}

We recall that a function $\mathfrak{n}_{Z}\colon Z\to[0,\infty)$ is a compact seminorm on a normed space $Z$ if it is a seminorm and for any bounded sequence $z_k\in Z$ there exists	a Cauchy subsequence $z_{k_j}$ with respect to $\mathfrak{n}_Z$, that is,
$\mathfrak{n}_{Z}(z_{k_j}-z_{k_l})\to 0$ as $j,l\to\infty$.

\begin{defn}\label{defn:QUASISTABILITY}
We say that a semigroup $\{S(t)\colon t\geq 0\}$ on a metric space $(V,d)$ is \emph{quasi-stable on a set $B\subseteq V$ at  time $T>0$ with respect to a compact seminorm} $\mathfrak{n}_{Z}$ if there exist constants $\eta\in[0,1)$, $\kappa>0$
and a~map $K\colon B\to Z$ into some auxiliary normed space $Z$ such that
\begin{align}\label{e:DIMSEM1}
\norm{Kx-Ky}_{Z}&\leq \kappa d(x,y),\ x,y\in B,\\
\label{e:DIMSEM2}
d(S(T)x,S(T)y)&\leq\eta d(x,y)+\mathfrak{n}_{Z}(Kx-Ky),\ x,y\in B.
\end{align}
\end{defn}

Following the proof of \cite[Theorem 3.1.21]{chueshov}, we now show that the quasi-stability
of a~semigroup on a  positively invariant bounded absorbing set $\mathbf{B}$ implies the covering condition  \eqref{e:CONDBALLSSEMB} and hence, the existence of a $T$-discrete exponential attractor $\mathbf{M_0}$ if the semigroup is asymptotically closed. 
For a given $\sigma\in(0,1-\eta)$, the estimates for the fractal dimension of $\mathbf{M_0}$ and the global attractor $\mathbf{A}$ are expressed in terms of the maximal cardinality of $\frac{\sigma}{2\kappa}$-distinguishable subsets of the closed unit ball $\overline{B}^{Z}(0,1)=\{z\in Z\colon\norm{z}_{Z}\leq 1\}$ in $Z$ with respect to the pseudometric generated by the seminorm $\mathfrak{n}_{Z}$ which we denote by
$$\mathfrak{m}_{Z}\left(\tfrac{\sigma}{2\kappa}\right)=m_{\mathfrak{n}_Z}(\overline{B}^{Z}(0,1),\tfrac{\sigma}{2\kappa}).$$

\begin{thm}\label{thm:QUASI}
Let $\{S(t)\colon t\geq 0\}$ be a semigroup on a metric space $(V,d)$, $T>0$ and let $B$ be a nonempty bounded set such that $S(T)B\subseteq B$. If the semigroup is quasi-stable on $B$ at time $T$ with respect to a compact seminorm $\mathfrak{n}_{Z}$ and parameters $(\eta,\kappa)$, then for any $\sigma\in(0,1-\eta)$
the covering condition
\begin{equation}\label{e:CONDBALLSSEMB2}
N^{V}(S(kT)B,aq^{k})\leq b h^{k},\ k\in\N,\ k\geq k_0,
\end{equation}
for some $k_0\in\N$ and $a,b>0$,
is satisfied with  $q=\eta+\sigma$ and $h=\mathfrak{m}_{Z}\left(\tfrac{\sigma}{2\kappa}\right)$.
\end{thm}

\begin{proof}
We set
$R=\max\{\diam^{V}(B),1\}$.
Note that $(B,\rho)$ is a precompact pseudometric space with
$$\rho(x,y)=\mathfrak{n}_{Z}(Kx-Ky),\ x,y\in B,$$
since $\mathfrak{n}_{Z}$ is compact, \eqref{e:DIMSEM1} holds
and $B$ is bounded.
In order to apply Lemma~\ref{lem:FUNDAMENTAL}, for $\sigma>0$ we estimate from above
the quantity
$$\varsigma_{\rho}(B,\sigma)=\sup_{\eps>0}c_{\rho}(B,\eps,\sigma\eps),$$
where $c_{\rho}(\cdot,\cdot,\cdot)$ is defined in \eqref{e:DEFCRHO}.

We fix $\eps>0$ and $\emptyset\neq F\subseteq B$ with
$\diam^{V}(F)\leq 2\eps$.
Let $m_{F}=m_{\rho}(F,\sigma\eps)$ and $\{y_1,\ldots,y_{m_{F}}\}$
be the maximal $\sigma\eps$-distinguishable subset of $F$ in $(B,\rho)$.
We define  $z_{j}= Ky_j\in Z$, $j=1,\ldots, m_{F}$, and observe that
\begin{align}\label{eq:proofQS}
\mathfrak{n}_{Z}(z_j-z_l)\geq\sigma\eps\ \text{ for }\ 1\leq j,l\leq m_F,\ j\neq l.
\end{align}
Also, due to \eqref{e:DIMSEM1}, we obtain
\begin{equation}\label{e:DIFFERENCEZJ}
\norm{z_j-z_l}_{Z}\leq \kappa\diam^{V}(F)\leq 2\eps \kappa,\ 1\leq j,l\leq m_F.
\end{equation}
We now choose an arbitrary point $z_j$, denote it by $z_0$, and
note that \eqref{eq:proofQS} and \eqref{e:DIFFERENCEZJ} imply
$$\tfrac{1}{2\eps\kappa}(z_j-z_0)\in\overline{B}^{Z}(0,1),\ 1\leq j\leq m_F,$$
and
$$\mathfrak{n}_{Z}\left(\tfrac{1}{2\eps\kappa}(z_j-z_0)-\tfrac{1}{2\eps\kappa}(z_l-z_0)\right)\geq\tfrac{\sigma}{2\kappa}\ \text{ for }\ 1\leq j,l\leq m_F,\ j\neq l.$$
By the compactness of $\mathfrak{n}_{Z}$ the unit ball $\overline{B}^{Z}(0,1)$
is precompact in $(Z,\zeta)$ with the pseudometric
$$\zeta(w,z)=\mathfrak{n}_{Z}(w-z),\ w,z\in Z,$$
and thus $m_F$ is bounded from above by $m_{\zeta}\big(\overline{B}^{Z}(0,1),\frac{\sigma}{2\kappa}\big)=\mathfrak{m}_{Z}\left(\frac{\sigma}{2\kappa}\right)$.
This shows that for any nonempty $A\subseteq B$ and $\eps>0$ we have by \eqref{e:DEFCRHO}
$$c_{\rho}(A,\eps,\sigma\eps)\leq\varsigma_{\rho}(B,\sigma)\leq\mathfrak{m}_{Z}\left(\tfrac{\sigma}{2\kappa}\right).$$
We apply \eqref{e:DIMSEM2} and Lemma~\ref{lem:FUNDAMENTAL} with $S=S(T)$, $A=B$
and $\eps=\frac{1}{2}R$ to get with $q=\eta+\sigma$
$$\widehat{N}^{V}(S(T)B,\tfrac{q}{2}R)\leq\widehat{N}^{V}(B,\tfrac{1}{2}R)\mathfrak{m}_{Z}\left(\tfrac{\sigma}{2\kappa}\right)=\mathfrak{m}_{Z}\left(\tfrac{\sigma}{2\kappa}\right).$$
Now we can apply Lemma~\ref{lem:FUNDAMENTAL} with $S=S(T)$
and $A=S(T)B\subseteq B$ and $\eps=\frac{q}{2}R$ to get
$$\widehat{N}^{V}(S(2T)B, \tfrac{q^2}{2}R)\leq (\mathfrak{m}_{Z}\left(\tfrac{\sigma}{2\kappa}\right))^{2}.$$
Using Lemma~\ref{lem:FUNDAMENTAL} again, we obtain by induction for $k\in\N$
$$\widehat{N}^{V}(S(kT)B, \tfrac{q^{k}}{2} R)\leq
\left(\mathfrak{m}_{Z}\left(\tfrac{\sigma}{2\kappa}\right)\right)^{k}.$$
Since $N^{V}(A,3\eps)\leq\widehat{N}^{V}(A,\eps)$, we conclude that
the covering condition \eqref{e:CONDBALLSSEMB} is satisfied with $a=\frac{3}{2}R, b=1$, $q=\eta+\sigma$ and $h=\mathfrak{m}_{Z}\left(\frac{\sigma}{2\kappa}\right)$. 
\end{proof}

Combining Theorems~\ref{thm:EXDEXPAT} and \ref{thm:QUASI} with Remark~\ref{rem:POSINV}, we get the following existence result for $T$-discrete exponential attractors.

\begin{thm}
Let $\{S(t)\colon t\geq 0\}$ be an asymptotically closed semigroup on a~complete metric space $(V,d)$, $T>0$ and let $\mathbf{B}\subseteq V$ be a bounded absorbing set for the semigroup. If the semigroup is quasi-stable on $\mathbf{B}$ at time $T$ with respect to a~compact seminorm $\mathfrak{n}_{Z}$ and parameters $(\eta,\kappa)$, then for any $\sigma\in(0,1-\eta)$ there exists a~$T$-discrete exponential attractor $\mathbf{M_0}\subseteq\mathbf{B}$ in $V$ for the semigroup with rate of attraction $\xi\in(0,\frac{1}{T}\ln{\frac{1}{\eta+\sigma}})$, and its fractal dimension is bounded by
\begin{equation*}
\dimf^{V}(\mathbf{M_0})\leq\log_{\frac{1}{\eta+\sigma}}\mathfrak{m}_{Z}\left(\tfrac{\sigma}{2\kappa}\right).
\end{equation*}
Moreover, the semigroup has a global attractor $\mathbf{A}$  contained in $\mathbf{M_0}$.
\end{thm}

\section{Construction based on generalized smoothing property}\label{sec:CE}

In this section we address construction methods of exponential attractors for semigroups in Banach 
spaces that are based on compact embeddings. More specifically, we consider two classes of semigroups that are quasi-stable. 
The first proposition addresses semigroups considered by R.~Czaja and M.~Efendiev in \cite[Theorem 3.2]{CzEf}, and the second proposition semigroups that generalize the setting used by A.~N.~Carvalho and S.~Sonner in \cite{CaSo}. 
We say that the semigroups considered in Proposition \ref{prop:czaja} satisfy the \emph{generalized smoothing property}. 
These results provide sufficient conditions for quasi-stability and will be applied in subsequent sections to verify that semigroups are quasi-stable and possess $T$-discrete exponential attactors. 

\begin{prop}\label{prop:czaja}
Let $\{S(t)\colon t\geq 0\}$ be a semigroup in a metric space $(V,d)$, $T>0$ and $B$ be a~subset of $V$. 
Let $Y,Z$ be normed spaces such that $Z$ is compactly embedded into $Y$. Assume that there exists a map $M\colon B\to Z$ and parameters $\eta\in[0,1)$, $\mu>0$ and $\kappa>0$ such that for all $x,y\in B$
\begin{equation}\label{e:CZAJA1}
\|Mx-My\|_Z\leq \kappa d(x,y),
\end{equation}
\begin{equation}\label{e:CZAJA2}
d(S(T)x,S(T)y)\leq \eta d(x,y)+\mu\|Mx-My\|_Y
\end{equation} 
holds. Then, $\{S(t)\colon t\geq 0\}$ is quasi-stable on $B$ at time $T$ with parameters $(\eta,\kappa\mu)$ and the compact seminorm $\mathfrak{n}_{Z}(x)=\|x\|_Y$ on $Z$.
\end{prop}

\begin{proof}
We observe that 
$\mathfrak{n}_{Z}(x)=\|x\|_Y$ is a compact seminorm on $Z$ since $Z$ is compactly embedded into $Y$. 
Moreover, \eqref{e:CZAJA1} and \eqref{e:CZAJA2} imply that 
\begin{align*}
\|\mu Mx-\mu My\|_Z=\mu \|Mx-My\|_Z\leq \mu\kappa d(x,y).
\end{align*}
Hence, the semigroup is quasi-stable according to Definition \ref{defn:QUASISTABILITY} with $\mathfrak{n}_{Z}(x)=\|x\|_Y, K=\mu M$ and parameters $(\eta,\kappa\mu)$. 
\end{proof}

Considering in Proposition~\ref{prop:czaja} a nonempty subset $V$ of a normed space $X$ and taking $Z=X$, $M=\kappa I$ and $\mu=1$, we obtain the following result.

\begin{prop}\label{prop:czaja_mod}
Let $\{S(t)\colon t\geq 0\}$ be a semigroup on a nonempty subset $V$ of a normed space $X$, $T>0$ and $B$ be a subset of $V$. 
Let $Y$ be a normed space such that $X$ is compactly embedded into $Y$. Assume that there exist parameters $\eta\in[0,1)$ and $\kappa>0$ such that for all $x,y\in B$ 
\begin{align*}
\begin{split}
\|S(T)x-S(T)y\|_X&\leq \eta\|x-y\|_X+\kappa\|x-y\|_Y
\end{split}
\end{align*}
holds. Then, $\{S(t)\colon t\geq 0\}$ is quasi-stable on $B$ with parameters $(\eta,\kappa)$  and the compact seminorm $\mathfrak{n}_{X}(x)=\|x\|_Y$ on $X$.
\end{prop}

Note that the map $S(T)$ in Propositions~\ref{prop:czaja} and \ref{prop:czaja_mod} is Lipschitz continuous on $B$. Hence, by Remark~\ref{rem:closed} we can either assume asymptotic closedness of the semigroup or closedness of the absorbing set to conclude the existence of $T$-discrete  exponential attractors.
Combining Propositions~\ref{prop:czaja},~\ref{prop:czaja_mod} with Theorems~\ref{thm:EXDEXPAT},~\ref{thm:QUASI} and Remark~\ref{rem:closed}, we get the following theorem.

\begin{thm}\label{thm:expattr_CE}
Let $\{S(t)\colon t\geq 0\}$ be a semigroup on a nonempty closed subset $V$ of a Banach space $X$, $T>0$  and $\mathbf{B}\subseteq V$ be a bounded absorbing set for the semigroup. Moreover, let $\{S(t)\colon t\geq 0\}$ be asymptotically closed or $\mathbf{B}$ be closed. 

If the semigroup satisfies the hypotheses of Proposition \ref{prop:czaja} or Proposition \ref{prop:czaja_mod} on $\mathbf{B}$  then for any $\sigma\in(0,1-\eta)$ there exists a~$T$-discrete exponential attractor  $\mathbf{M_0}\subseteq \mathbf{B}$ for the semigroup and its fractal dimension is bounded by 
\begin{equation*}
\dimf^{V}(\mathbf{M_0})\leq 
\begin{cases}
\log_{\frac{1}{\eta+\sigma}}m_{\|\cdot\|_Y}\left(\overline{B}^Z(0,1),\frac{\sigma}{2\mu\kappa}\right)& \text{if the hypotheses of Proposition \ref{prop:czaja} hold},\\
\log_{\frac{1}{\eta+\sigma}}m_{\|\cdot\|_Y}\left(\overline{B}^X(0,1),\frac{\sigma}{2\kappa}\right)& \text{if the hypotheses of Proposition \ref{prop:czaja_mod} hold}.
\end{cases}
\end{equation*}
If the semigroup is asymptotically closed, it has a global attractor $\mathbf{A}$  contained in $\mathbf{M_0}$.
\end{thm}

Theorem \ref{thm:expattr_CE} generalizes existence results for exponential attractors in \cite{CzEf} and 
\cite{CaSo}. In particular, using the concept of quasi-stability allows that the constant $\eta\in[0,1)$ while in previous constructions it was assumed that $\eta\in[0,1/2),$ see also the remarks at the end of Section \ref{sec:SP}.

\section{Construction based on smoothing property}\label{sec:SP}

In \cite{EfMiZe} M.~Efendiev, A.~Miranville and S.~Zelik applied the smoothing property to construct exponential attractors for semigroups in Banach spaces. We show that the smoothing property implies the generalized smoothing property and hence, quasi-stability of the semigroup which allows us to generalize previous existence results for exponential attractors using the smoothing property. 
The smoothing property is based on the Lipschitz continuity of an operator between two normed spaces and the compact embedding of these spaces. In the following definition we include two different settings.  
Either the compactly embedded space is a subspace of the phase space, or the phase space is compactly embedded into an auxiliary normed space. Both cases lead to estimates for the fractal dimension of the exponential attractor which are determined by the $\eps$-capacity properties of the compact embedding, see Theorem \ref{thm:expattr_smooth} and compare to e.g. \cite{EfMiZe,CCM,CaSo,sonner2012}.   

\begin{defn}\label{defn:smooth}
We say that a semigroup $\{S(t)\colon t\geq0\}$ on a nonempty subset $V$ of a~normed space $(X,\norm{\cdot}_{X})$
satisfies the \textit{smoothing property on a subset $B$ of $V$ at time $T>0$ with parameters $(\eta,\kappa)$} if $S(T)=C(T)+M(T)$, where 
the map $C(T)$ is a contraction in $X$, i.e., there exists $\eta\in [0,1)$ such that
\begin{equation}\label{eq:contraction}
\norm{C(T)x-C(T)y}_{X}\leq \eta\norm{x-y}_{X},\ x,y\in B,
\end{equation}
and one of the following two properties holds:
\begin{itemize}
\item[(i)] 
$M(T)\colon B\to Z$, where $Z$ is an auxiliary normed space compactly embedded into $X$, and there exists $\kappa>0$ such that 
\begin{equation}\label{eq:smoothingRadek}
\norm{M(T)x-M(T)y}_{Z}\leq \kappa\norm{x-y}_{X},\ x,y\in B,
\end{equation} 
\item[(ii)] 
$X$ is compactly embedded into a normed space $Y$ and there exists $\kappa>0$  such that 
 \begin{equation}\label{eq:smoothingStefanie}
\norm{M(T)x-M(T)y}_{X}\leq \kappa\norm{x-y}_{Y},\ x,y\in B.
\end{equation} 
\end{itemize}
\end{defn}

\begin{prop}\label{prop:smoothqs}
Let $\{S(t)\colon t\geq 0\}$ be a semigroup on a nonempty subset $V$ of a normed space $(X,\norm{\cdot}_{X})$, $T>0$ 
and let $B$ be a bounded subset of $V$.
If the semigroup satisfies the smoothing property on $B$ at time $T$ with parameters $(\eta,\kappa)$, then it satisfies the generalized smoothing property and hence, it is quasi-stable on $B$ at time $T$ with parameters $(\eta,\kappa)$.

Moreover, if $S(T)B\subseteq B$ then for any $\sigma\in (0,1-\eta)$ the covering condition \eqref{e:CONDBALLSSEMB2} holds with $q=\eta+\sigma$ and
\begin{align*}
h&=
\begin{cases}
m_{\|\cdot\|_X}\left(\overline{B}^Z(0,1),\frac{\sigma}{2\kappa}\right)& \text{if \eqref{eq:smoothingRadek} holds},\\
m_{\|\cdot\|_Y}\left(\overline{B}^X(0,1),\frac{\sigma}{2\kappa}\right)& \text{if \eqref{eq:smoothingStefanie} holds}.
\end{cases}
\end{align*}
\end{prop}

\begin{proof}
Let $x,y\in B$ and assume that the smoothing property holds with \eqref{eq:smoothingRadek}. Then, using \eqref{eq:contraction} we obtain 
\begin{align*}
\|S(T)x-S(T)y\|_X&\leq \|C(T)x-C(T)y\|_X+\|M(T)x-M(T)y\|_X\\
&\leq \eta\|x-y\|_X+\|M(T)x-M(T)y\|_X.
\end{align*}
Moreover, by \eqref{eq:smoothingRadek} $Z$ is compactly embedded into $X$ and
\begin{align*}
\|M(T)x-M(T)y\|_Z\leq 
\kappa\|x-y\|_X.
\end{align*} 
This shows that the hypotheses of Proposition~\ref{prop:czaja} are satisfied with $Y=X$, $\mu=1$ and $M=M(T)$ and hence, the semigroup is quasi-stable with parameters $(\eta,\kappa)$ and the compact seminorm $\mathfrak{n}_{Z}(x)=\|x\|_X$ on $Z$. 

Let now $x,y\in B$ and assume that the smoothing property holds with \eqref{eq:smoothingStefanie}. Then, we get
\begin{align*}
\|S(T)x-S(T)y\|_X&\leq \|C(T)x-C(T)y\|_X+\|M(T)x-M(T)y\|_X\\
&\leq \eta\|x-y\|_X+\kappa\|x-y\|_Y.
\end{align*}
Hence, the hypotheses of Proposition~\ref{prop:czaja_mod} are satisfied and we conclude that 
the semigroup is quasi-stable with parameters $(\eta,\kappa)$ and the compact seminorm $\mathfrak{n}_{X}(x)=\|x\|_Y$ on $X$.
 
Finally, Theorem \ref{thm:QUASI} implies that for any $\sigma\in (0,1-\eta)$ the covering condition holds with the stated parameters $q$ and $h$.
\end{proof}

Combining Proposition \ref{prop:smoothqs}, Theorems~\ref{thm:EXDEXPAT},~\ref{thm:QUASI} and Remarks~\ref{rem:POSINV},~\ref{rem:closed}, we obtain the following existence result for $T$-discrete exponential attractors. 

\begin{thm}\label{thm:expattr_smooth}
Let $\{S(t)\colon t\geq 0\}$ be a semigroup on a nonempty closed subset $V$ of a Banach space $(X,\norm{\cdot}_{X})$, $T>0$  and $\mathbf{B}\subseteq V$ be a bounded absorbing set for the semigroup. Moreover, let $\{S(t)\colon t\geq 0\}$ be asymptotically closed or $\mathbf{B}$ be closed. 

If the semigroup satisfies the smoothing property on $\mathbf{B}$ at time $T$ with parameters $(\eta,\kappa)$, then for any $\sigma\in(0,1-\eta)$ there exists a~$T$-discrete exponential attractor  $\mathbf{M_0}\subseteq \mathbf{B}$ for the semigroup and
\begin{equation*}
\dimf^{V}(\mathbf{M_0})\leq 
\begin{cases}
\log_{\frac{1}{\eta+\sigma}}m_{\|\cdot\|_X}\left(\overline{B}^Z(0,1),\frac{\sigma}{2\kappa}\right)& \text{if \eqref{eq:smoothingRadek} holds},\\
\log_{\frac{1}{\eta+\sigma}}m_{\|\cdot\|_Y}\left(\overline{B}^X(0,1),\frac{\sigma}{2\kappa}\right)& \text{if \eqref{eq:smoothingStefanie} holds}.
\end{cases}
\end{equation*}
If the semigroup is asymptotically closed, then it has a global attractor $\mathbf{A}$ contained in $\mathbf{M_0}$.
\end{thm}

\begin{rem}
In previous constructions of exponential attractors based on the smoothing property it was assumed in \eqref{eq:contraction} that the contraction rate $\eta\in[0,1/2)$, see e.g. \cite{EfMiZe, CzEf, CCM, CaSo}. This assumption can be weakened to $\eta\in[0,1)$ by using the framework of quasi-stability and Lemma \ref{lem:FUNDAMENTAL} which is based on minimal coverings by sets of a certain diameter. On the contrary, in \cite{EfMiZe, CzEf, CCM, CaSo} the smoothing property is used to construct coverings of iterates of the absorbing set under the time evolution of the semigroup by balls with centers that lie in the set which requires the more restrictive assumption $\eta\in[0,1/2)$, see also \cite{chueshov}.

The estimate for the fractal dimension of the exponential attractor in Theorem \ref{thm:expattr_smooth} is determined by properties of the compact embedding of the spaces $Z$ and $X$, and $X$ and $Y$, respectively. We recall that $m_{\|\cdot\|_Y}\left(\overline{B}^X(0,1),\frac{\sigma}{2\kappa}\right)$ denotes the maximal cardinality of subsets of $\overline{B}^X(0,1)$ that are $\frac{\sigma}{2\kappa}$-distinguishable in $Y$. 
If $X$ is compactly embedded into $Y$ the $\varepsilon$-capacity of this embedding is defined as 
$$\mathcal{C}_\varepsilon(X,Y)=\log_2 \left(m_{\|\cdot\|_Y}\left(\overline{B}^X(0,1),\varepsilon\right)\right),$$
which is closely related to the $\varepsilon$-entropy of the embedding,
$$\mathcal{H}_\varepsilon(X,Y)=\log_2 \left(N^Y\left(\overline{B}^X(0,1),\varepsilon\right)\right).$$
These concepts were introduced by A.~N.~Kolmogorov and V.~M.~Tikhomirov in \cite{KoTi}. 
For certain function spaces explicit estimates are known for the $\varepsilon$-capacity and $\varepsilon$-entropy which would yield explicit estimates of the fractal dimension of the exponential attractor in Theorem \ref{thm:expattr_smooth}. While the constructions of exponential attractors in \cite{EfMiZe, CzEf, CCM, CaSo, sonner2012} lead to estimates for the fractal dimension that are determined by the $\varepsilon$-entropy of the corresponding embedding, here we use the quasi-stability and obtain estimates in terms of the $\varepsilon$-capacity. 
\end{rem}

\section{Construction based on squeezing property}\label{sec:SQ}

In this section, we discuss semigroups in Banach spaces satisfying a squeezing property with respect to a given finite-dimensional normed space. Originally, the squeezing property was considered for semigroups in Hilbert spaces using an orthogonal projection $P$ onto a~finite-dimensional subspace. C. Foias and R. Temam first introduced the squeezing property in \cite{FoTe0}, see also \cite{EFNT94} and \cite[p.~101]{chueshov}, and it was the method applied in the first existence proof of exponential attractors in \cite{EFNT94}. 

Here, we introduce a generalized squeezing property and extend earlier approaches to a~Banach space setting. 
In particular, we will consider a map $P\colon X\to X_n$, possibly nonlinear, with values in a finite-dimensional space $X_n$. This setting allows metric projections in uniformly convex Banach spaces, see \cite[p.~392]{Brezis11}, as well as bounded linear operators $P\in\mathcal{L}(X,X_n)$, including orthogonal projections in Hilbert spaces.
We will show that semigroups satisfying the generalized squeezing property are quasi-stable. Consequently, the covering condition \eqref{e:CONDBALLSSEMB} holds which implies the existence of a $T$-discrete exponential attractor. The estimates for its fractal dimension are determined by the parameters in the condition for quasi-stability. However, exploiting the structure of the generalized squeezing property we can improve these estimates and get better bounds on the fractal dimension than the ones obtained in previous sections.  

\begin{defn}\label{defn:BANACHSQU}
We say that a semigroup $\{S(t)\colon t\geq 0\}$ on a nonempty subset $V$ of a normed space $(X,\norm{\cdot}_{X})$
is \emph{squeezing} on a subset $B$ of $V$ at positive time $T>0$ with parameters $(n,\eta,\mu,\kappa)$ if
\begin{itemize}
\item[(a)] there exists a map $P\colon V\to X_n$ with values in a normed space $(X_n,\norm{\cdot}_{X_n})$ of dimension $n\in\N$ and constants  $\eta\in[0,1)$ and $\mu>0$ such that for any $x,y\in B$
$$\norm{S(T)x-S(T)y}_{X}>\mu\norm{PS(T)x-PS(T)y}_{X_n}$$
implies that
$$\norm{S(T)x-S(T)y}_{X}\leq\eta\norm{x-y}_{X},$$
\item[(b)] $PS(T)$ is Lipschitz continuous on $B$ with Lipschitz constant $\kappa>0$, i.e.,
\begin{equation*}
\norm{PS(T)x-PS(T)y}_{X_n}\leq \kappa\norm{x-y}_{X},\ x,y\in B.
\end{equation*}
\end{itemize}
Moreover, we say that a semigroup $\{S(t)\colon t\geq 0\}$ satisfies the \emph{generalized squeezing property} with parameters $(n,\eta,\mu,\kappa)$ if there exist $P$, $X_n$, $\mu$ and $\eta$ as in (a) such that 
\begin{equation}\label{e:BANACHSQU2}
\norm{S(T)x-S(T)y}_{X}\leq\eta\norm{x-y}_{X}+\mu\norm{PS(T)x-PS(T)y}_{X_n},\ x,y\in B,
\end{equation}
and (b) holds.
\end{defn}

\begin{rem}
The implication in (a) can be equivalently stated as an alternative: for any $x,y\in B$ either
\begin{equation*}
\norm{S(T)x-S(T)y}_{X}\leq\mu\norm{PS(T)x-PS(T)y}_{X_n}
\end{equation*}
holds or
\begin{equation*}
\norm{S(T)x-S(T)y}_{X}\leq\eta\norm{x-y}_{X}.
\end{equation*}
\end{rem}

Since in the above definitions $X_n$ is a general $n$-dimensional normed space, we will compare it with $\ell^n_2$, the space $\K^n$, $\mathbb{K}\in\{\R,\C\}$, endowed with the Euclidean norm $|\cdot|_2$, using the notion of the \emph{multiplicative Banach-Mazur distance} (see \cite[II.E.6]{wojtaszczyk}), 
\begin{equation*}
d_{BM}(X_n,\ell^n_2)=\inf\{\|\Lambda\|_{\mathcal{L}(X_n,\ell^n_2)}\|\Lambda^{-1}\|_{\mathcal{L}(\ell^n_2,X_n)}:\  \Lambda\colon X_n\to\ell^n_2\text{ is a linear isomorphism}\}.
\end{equation*} 
Recall from F.~John's Theorem (see \cite[Corollary III.B.9]{wojtaszczyk}) that
\begin{equation}\label{e:JOHN}
d_{BM}(X_n,\ell^n_2)\leq\sqrt{n},
\end{equation} 
where equality holds, for example, for $X_n=\ell^{n}_\infty$ or $X_n=\ell^{n}_1$, i.e., $\K^n$ endowed with the maximum norm or the norm of the sum of absolute values, respectively; see \cite[Proposition II.E.8]{wojtaszczyk}. Moreover, if $X_n$ is an $n$-dimensional subspace of a Hilbert space $X$, then $d_{BM}(X_n,\ell^n_2)=1$, since $X_n$ is then isometric to $\ell^n_2$. In this case, for $P$ one can take an orthogonal projection of $X$ onto $X_n$ with $\norm{P}_{\mathcal{L}(X,X_n)}=1$. 

Obviously, the squeezing property implies the generalized squeezing property. We now show that this in turn implies the quasi-stability of the semigroup and hence, the existence of a $T$-discrete exponential attractor. We obtain the estimate \eqref{e:CONDBALLSSEMB} in terms of the dimension $n$ of the space $X_n$ by the comparison of volumes in the Euclidean space.

\begin{prop}\label{prop:SQ2QUASIBANACH}
If  $\{S(t)\colon t\geq 0\}$ is a semigroup on a nonempty subset $V$ of a normed space $(X,\norm{\cdot}_{X})$ over $\mathbb{K}\in\{\R,\C\}$
satisfying the generalized squeezing property on a subset $B$ of $V$ at time $T>0$ with parameters $(n,\eta,\mu,\kappa)$, then it is quasi-stable on $B$ at time $T$ with parameters $(\eta,\kappa\mu)$ with respect to the norm $\norm{\cdot}_{X_n}$ on $X_{n}$. Moreover, if $B$ is a nonempty bounded subset of $V$ such that $S(T)B\subseteq B$ and $\sigma\in(0,1-\eta)$, then the covering condition \eqref{e:CONDBALLSSEMB2} holds with $q=\eta+\sigma$ and
\begin{equation}\label{e:SQUEEZINGWORSE}
h\leq\Bigl(1+\tfrac{4\kappa\mu d_{BM}(X_n,\ell^n_2)}{\sigma}\Bigr)^\mathbf{n},
\end{equation}
where
\begin{equation}\label{e:ELL}
\mathbf{n}=\begin{cases}
n& \text{ if }\ \mathbb{K}=\R,\\
2n&  \text{ if }\ \mathbb{K}=\C.
\end{cases}
\end{equation}
\end{prop}

\begin{proof}
If the generalized squeezing property holds, the  hypotheses of Proposition \ref{prop:czaja} are satisfied with 
$$Z=X_n, \quad Y=X_n,\quad M=PS(T),$$ 
which implies that the semigroup is quasi-stable with parameters $(\eta,\kappa\mu)$ with respect to the norm $\norm{\cdot}_{X_n}$ on $X_{n}$.

By Theorem \ref{thm:QUASI}, for any $\sigma\in(0,1-\eta)$, the covering condition \eqref{e:CONDBALLSSEMB2} holds with $q=\eta+\sigma$ and 
$$h=m_{\norm{\cdot}_{X_n}}\left(\overline{B}^{X_n}(0,1),\tfrac{\sigma}{2\kappa\mu}\right).$$
Let $\Lambda\colon X_n\to \ell^n_2$ be a linear isomorphism between $X_n$ and $\ell^n_2$.
Let $x_1,\ldots,x_{h}\in \overline{B}^{X_n}(0,1)$ be such that $\norm{x_j-x_l}_{X_n}\geq\frac{\sigma}{2\kappa\mu}$ for $j\neq l$. Considering $z_j=\Lambda x_j$, $j=1,\ldots,h$, we see that $|z_j|_2\leq\norm{\Lambda}_{\mathcal{L}(X_n,\ell^n_2)}$
and 
$$|z_j-z_l|_2\geq\frac{\sigma}{2\kappa\mu\|\Lambda^{-1}\|_{\mathcal{L}(\ell^n_2,X_n)}}\ \text{ for}\ j\neq l.$$ 
Therefore, we get 
\begin{equation*}
h \leq m_{|\cdot|_2}\Bigl(\overline{B}^{\R^{\mathbf{n}}}\Big(0,\|\Lambda\|_{\mathcal{L}(X_n,\ell^n_2)}\Big),\frac{\sigma}{2\kappa\mu\|\Lambda^{-1}\|_{\mathcal{L}(\ell^n_2,X_n)}}\Bigr),   
\end{equation*}
where $\Lambda\colon X_n\to \ell^{n}_2$ is any linear isomorphism between $X_n$ and $\ell^n_2$ and $\mathbf{n}$ is given in \eqref{e:ELL}.

To shorten the notation we introduce
$$m=m_{|\cdot|_2}\Bigl(\overline{B}^{\R^\mathbf{n}}(0,r),\eps\Bigr)$$
and estimate $m$ from above following \cite[Lemma 3.1.4]{chueshov}. There exist points $y_1,\ldots,y_m\in \overline{B}^{\R^\mathbf{n}}(0,r)$  such that $|y_j-y_l|_{2}\geq\eps$ for $j\neq l$. Hence
$$\bigcup_{j=1}^{m}B^{\R^\mathbf{n}}\left(y_j,\frac{\eps}{2}\right)\subseteq B^{\R^\mathbf{n}}\left(0,r+\frac{\eps}{2}\right)$$
and the comparison of volumes yields 
\begin{equation}\label{e:BOUNDOFM}
m\leq\left(1+\frac{2r}{\eps}\right)^\mathbf{n}.
\end{equation}
Taking $\eps=\frac{\sigma}{2\kappa\mu\|\Lambda^{-1}\|_{\mathcal{L}(\ell^n_2,X_n)}}$
and $r=\|\Lambda\|_{\mathcal{L}(X_n,\ell^n_2)}$, we obtain \eqref{e:SQUEEZINGWORSE}.
\end{proof}

We can improve the estimate  for $h$ in Proposition~\ref{prop:SQ2QUASIBANACH} by exploiting the squeezing property, the Bieberbach-Urysohn isodiametric inequality, and Lemma \ref{lem:FUNDAMENTAL} in order to improve the estimates for the fractal dimension of the exponential attractor. 

\begin{prop}\label{prop:QUASISQIMPROVEMENT}
Under the assumptions of Proposition~\ref{prop:SQ2QUASIBANACH}, for any $\sigma\in(0,1-\eta)$, the covering condition \eqref{e:CONDBALLSSEMB2} holds with $q=\eta+\sigma$ and
\begin{equation}\label{e:SQUEEZINGGOOD}
h\leq\left(1+\tfrac{2\kappa\mu d_{BM}(X_n,\ell^n_2)}{\sigma}\right)^{\mathbf{n}},
\end{equation}
where $\mathbf{n}$ is given in \eqref{e:ELL}. 
\end{prop}

\begin{proof}
By Proposition~\ref{prop:SQ2QUASIBANACH} the semigroup is quasi-stable on $B$ at time $T$ with parameters $(\eta,\kappa\mu)$
with respect to $K=\mu PS(T)$ and the compact norm $\norm{\cdot}_{X_n}$ on $X_n$.
The proof follows the lines of the proof of Theorem \ref{thm:QUASI} with the precompact pseudometric space $(B,\rho)$ where
$$\rho(x,y)=\norm{Kx-Ky}_{X_n}=\mu\norm{PS(T)x-PS(t)y}_{X_{n}},\ x,y\in B.$$
In order to apply Lemma~\ref{lem:FUNDAMENTAL}, for $\sigma>0$ we improve the estimate for 
$$\varsigma_{\rho}(B,\sigma)=\sup_{\eps>0}c_{\rho}(B,\eps,\sigma\eps),$$
where $c_{\rho}(\cdot,\cdot,\cdot)$ is defined in \eqref{e:DEFCRHO}.

Let $\eps>0$, $\emptyset\neq F\subseteq B$ with $\diam^{X}(F)\leq 2\eps$, and $\{y_1,\ldots,y_{m_{F}}\}\subseteq F$ be a  maximal $\sigma\eps$-distinguishable subset of $F$ in $(B,\rho)$, where  $m_{F}=m_{\rho}(F,\sigma\eps)$. Setting $z_j=S(T)y_j\in B$, $j=1,\ldots,m_{F}$, we see that
$$\norm{Pz_j-Pz_l}_{X_n}\geq\tfrac{\sigma\eps}{\mu},\ j\neq l.$$
Fix an arbitrary linear isomorphism $\Lambda\colon X_{n}\to\ell^{n}_2$ and let $R\colon\ell^n_2\to\R^{\mathbf{n}}$ denote the (real) isometry. Let 
$\{x_1,\ldots,x_{m}\}$ be a maximal $\frac{\sigma\eps}{\mu\norm{\Lambda^{-1}}_{\mathcal{L}(\ell^n_2,X_n)}}$-distinguishable subset of $R(\Lambda(P(S(T)F)))$, where
$$m=m_{|\cdot|_{2}}(R(\Lambda(P(S(T)F))),\tfrac{\sigma\eps}{\mu\norm{\Lambda^{-1}}_{\mathcal{L}(\ell^n_2,X_n)}}).$$ 
Note that $m_{F}\leq m$,
since the points
$$R(\Lambda(Pz_j))\in R(\Lambda(P(S(T)F))),\ j=1,\ldots,m_{F},$$ 
form a $\frac{\sigma\eps}{\mu\norm{\Lambda^{-1}}_{\mathcal{L}(\ell^n_2,X_n)}}$-distinguishable set.
Observe that the balls $B^{\R^{\mathbf{n}}}(x_j,\frac{\sigma\eps}{2\mu\norm{\Lambda^{-1}}_{\mathcal{L}(\ell^n_2,X_n)}})$, $j=1,\ldots,m$, are disjoint in $\R^\mathbf{n}$. Moreover, we have
\begin{equation}\label{e:FORVOLUMES}
\bigcup_{j=1}^{m}B^{\R^\mathbf{n}}\Big(x_j,\tfrac{\sigma\eps}{2\mu\norm{\Lambda^{-1}}_{\mathcal{L}(\ell^n_2,X_n)}}\Big)\subseteq B^{\R^\mathbf{n}}\Big(R(\Lambda(P(S(T)F))),\tfrac{\sigma\eps}{2\mu\norm{\Lambda^{-1}}_{\mathcal{L}(\ell^n_2,X_n)}}\Big),
\end{equation}
where the latter set means $\{x\in\R^\mathbf{n}\colon \dist^{\R^\textbf{n}}(x,R(\Lambda(P(S(T)F))))<\frac{\sigma\eps}{2\mu\norm{\Lambda^{-1}}_{\mathcal{L}(\ell^n_2,X_n)}}\}$.
Note that by (b) we have
$$\diam^{\R^{\mathbf{n}}}\Big(B^{\R^\mathbf{n}}\Big(R(\Lambda(P(S(T)F))),\tfrac{\sigma\eps}{2\mu\norm{\Lambda^{-1}}_{\mathcal{L}(\ell^n_2,X_n)}}\Big)\Big)\leq 2\eps\kappa\norm{\Lambda}_{\mathcal{L}(X_n,\ell^n_2)}+\tfrac{\sigma\eps}{\mu\norm{\Lambda^{-1}}_{\mathcal{L}(\ell^n_2,X_n)}}.$$
Indeed, for $v_i\in B^{\R^\mathbf{n}}(R(\Lambda(P(S(T)F))),\frac{\sigma\eps}{2\mu\norm{\Lambda^{-1}}_{\mathcal{L}(\ell^n_2,X_n)}}))$, $i=1,2$, there exist $w_i\in F\subseteq B$ such that
$$\abs{v_i-R(\Lambda(P(S(T)w_i)))}_{2}<\tfrac{\sigma\eps}{2\mu\norm{\Lambda^{-1}}_{\mathcal{L}(\ell^n_2,X_n)}},\ i=1,2.$$
By isometry of $R$ we get
\begin{equation*}
\begin{split}
\abs{v_1-v_2}_{2}&\leq \tfrac{\sigma\eps}{\mu\norm{\Lambda^{-1}}_{\mathcal{L}(\ell^n_2,X_n)}}+\abs{\Lambda(PS(T)w_1-PS(T)w_2)}_{2}\leq\tfrac{\sigma\eps}{\mu\norm{\Lambda^{-1}}_{\mathcal{L}(\ell^n_2,X_n)}}+2\eps\kappa\norm{\Lambda}_{\mathcal{L}(X_n,\ell^n_2)}.
\end{split}
\end{equation*}
We compare the volumes of sets in \eqref{e:FORVOLUMES} using the isodiametric inequality (see e.g. \cite[Theorem 2.4]{EG15}) and obtain
$$m\omega_{\mathbf{n}} \left(\tfrac{\sigma\eps}{2\mu\norm{\Lambda^{-1}}_{\mathcal{L}(\ell^n_2,X_n)}}\right)^{\mathbf{n}}\leq\omega_{\mathbf{n}}\left(\eps\kappa\norm{\Lambda}_{\mathcal{L}(X_n,\ell^n_2)}+\tfrac{\sigma\eps}{2\mu\norm{\Lambda^{-1}}_{\mathcal{L}(\ell^n_2,X_n)}}\right)^{\mathbf{n}},$$
where $\omega_{\mathbf{n}}=\frac{\pi^{\frac{\mathbf{n}}{2}}}{\Gamma(\frac{\mathbf{n}}{2}+1)}$ denotes the volume of the unit ball in $\R^{\mathbf{n}}$.
Consequently, we get
\begin{equation*}
m_{F}\leq m\leq \left(1+\tfrac{2\kappa\mu d_{BM}(X_n,\ell^n_2)}{\sigma}\right)^{\mathbf{n}},
\end{equation*}
and finally, we apply Lemma~\ref{lem:FUNDAMENTAL} as in the proof of Theorem~\ref{thm:QUASI}.
\end{proof}

Combining Proposition \ref{prop:QUASISQIMPROVEMENT} with Theorem \ref{thm:EXDEXPAT} and Remark \ref{rem:closed} we obtain the following existence theorem for $T$-discrete exponential attractors with an improved bound for the fractal dimension.

\begin{thm}\label{thm:SQEXPIMPROVED}
Let $\{S(t)\colon t\geq 0\}$ be a semigroup on a nonempty closed subset $V$ of a Banach space $(X,\norm{\cdot}_{X})$ over $\mathbb{K}\in\{\R,\C\}$ and $\mathbf{B}\subseteq V$ be a bounded absorbing set for the semigroup. Moreover, let $\{S(t)\colon t\geq 0\}$ be asymptotically closed or $\mathbf{B}$ be closed. 

If the semigroup satisfies the generalized squeezing property on $\mathbf{B}$ at time $T$ with parameters $(n,\eta,\mu,\kappa)$,
then for any $\sigma\in (0,1-\eta)$ there exists a~$T$-discrete exponential attractor $\mathbf{M_0}\subseteq \mathbf{B}$ for the semigroup and 
\begin{equation}\label{e:SQUEEZING}
\dimf^{V}(\mathbf{M_0})\leq \mathbf{n}\log_{\frac{1}{\eta+\sigma}}\left(1+\tfrac{2\kappa\mu d_{BM}(X_n,\ell^n_2)}{\sigma}\right),
\end{equation}
with $\mathbf{n}$ given in \eqref{e:ELL}. If the semigroup is asymptotically closed, then it has a global attractor $\mathbf{A}$  contained in $\mathbf{M_0}$.
\end{thm}

In the classical setting for squeezing semigroups in \emph{Hilbert spaces} $X$, the finite-dimensional space $X_n$ is an $n$-dimensional subspace of $X$ and $P$ is the orthogonal projection onto $X_n$. Thus $d_{BM}(X_n,\ell^n_2)=1$ and $\norm{P}_{\mathcal{L}(X,X_n)}=1$ and the estimates in  
\eqref{e:SQUEEZINGWORSE}, \eqref{e:SQUEEZINGGOOD}, and \eqref{e:SQUEEZING} simplify accordingly.
We obtain the following corollary. 

\begin{cor}\label{cor:6.6}
Let the assumptions of Theorem~\ref{thm:SQEXPIMPROVED} hold for a squeezing semigroup on a~nonempty closed subset $V$ of a~Hilbert space $(X,\norm{\cdot}_{X})$ and let  $X_n$ be an $n$-dimensional subspace of $X$ and $P$ be the orthogonal projection of $X$ onto $X_n$. Then, for any $\sigma\in (0,1-\eta)$ there exists a~$T$-discrete exponential attractor $\mathbf{M_0}\subseteq \mathbf{B}$ for the semigroup and 
\begin{equation*}
\dimf^{V}(\mathbf{M_0})\leq \mathbf{n}\log_{\frac{1}{\eta+\sigma}}\left(1+\tfrac{2\kappa\mu}{\sigma}\right),
\end{equation*}
with $\mathbf{n}$ given in \eqref{e:ELL}. If the semigroup is asymptotically closed, then it has a global attractor $\mathbf{A}$  contained in $\mathbf{M_0}$.
\end{cor}

\begin{rem}
Remaining in the \emph{Hilbert} setting of Corollary~\ref{cor:6.6}, note that if $\mu\in(0,1)$ in (a) of Definition~\ref{defn:BANACHSQU} then 
for any $x,y\in B$ such that $S(T)x-S(T)y\neq 0$ we have 
$$\norm{P(S(T)x-S(T)y)}_{X}\leq\norm{S(T)x-S(T)y}_{X}<\tfrac{1}{\mu}\norm{S(T)x-S(T)y}_{X}.$$
It follows that 
$$\norm{S(T)x-S(T)y}_{X}\leq\eta\norm{x-y}_{X},\ x,y\in B,$$
that is, $S(T)$ is a contraction on $B$.
Therefore, the only interesting case is when $\mu\geq 1$ in (a).  In fact, the statement (a) was originally written with $\mu=\sqrt{1+\alpha^2}$ for some $\alpha>0$. Then,  the implication can also be  expressed as follows:
there exist $\alpha>0$ and $\eta\in[0,1)$ such that for any $x,y\in B$ 
\begin{equation*}
\alpha\norm{P(S(T)x-S(T)y)}_{X}<\norm{(I-P)(S(T)x-S(T)y)}_{X},
\end{equation*}
implies that 
$$\norm{S(T)x-S(T)y}_{X}\leq\eta\norm{x-y}_{X}.$$
Equivalently, for any $x,y\in B$
either
\begin{equation*}
\norm{(I-P)(S(T)x-S(T)y)}_{X}\leq\alpha\norm{P(S(T)x-S(T)y)}_{X}
\end{equation*}
holds or
\begin{equation*}
\norm{S(T)x-S(T)y}_{X}\leq\eta\norm{x-y}_{X}.
\end{equation*}
\end{rem}

\section{Construction for semigroups of Ladyzhenskaya type}\label{sec:Lad}

A special type of squeezing semigroups was used by O. Ladyzhenskaya \cite{Lad82b}
in 1982 to estimate the fractal dimension of the global attractor for the 2D Navier-Stokes equation. This notion was later further investigated and the estimates for the fractal dimension were improved, e.g. in \cite[Theorems 3 and 4]{BV99}. Here, we introduce the notion of Ladyzhenskaya type semigroups in normed spaces, although the classical setting is in a Hilbert space with an orthogonal projection $P$ onto a finite-dimensional subspace. We show that these semigroups form a subclass of quasi-stable semigroups by comparing them with squeezing semigroups and smoothing semigroups. By exploiting the Hilbert space structure of the phase space we also improve the bounds for the fractal dimension of $T$-discrete exponential attractors obtained for squeezing semigroups in the previous section.

\begin{defn}\label{defn:LADY}
We say that a semigroup $\{S(t)\colon t\geq0\}$ on a nonempty subset $V$ of a normed space $(X,\norm{\cdot}_{X})$ over $\K\in\{\R,\C\}$
is of \emph{Ladyzhenskaya type} on a subset $B$ of $V$ at time $T>0$ with parameters $(n,\eta,\kappa)$ if
\begin{itemize}
\item[(a)] there exists a subspace $X_n$ of $X$ of dimension $n\in\N$, a map $P\colon V\to X_n$ and
a constant $\eta\in[0,1)$ such that
\begin{equation}\label{e:LADY}
\norm{(I-P)S(T)x-(I-P)S(T)y}_{X}\leq\eta\norm{x-y}_{X},\ x,y\in B,
\end{equation} 
\item[(b)] $PS(T)$ is Lipschitz continuous on $B$ with Lipschitz constant $\kappa>0$, i.e., 
\begin{equation*}
\norm{PS(T)x-PS(T)y}_{X}\leq \kappa\norm{x-y}_{X},\ x,y\in B.
\end{equation*}
\end{itemize}
\end{defn}

\begin{rem}
Comparing the notion of Ladyzhenskaya type semigroups with the squeezing property we observe the following. 
\begin{itemize}
\item[(i)] 
The condition (a) in Definition~\ref{defn:LADY} implies property (a) in Definition~\ref{defn:BANACHSQU} with $\mu=1+\alpha$ for $\alpha>0$ so large that $(1+\frac{1}{\alpha})\eta<1$. Thus, semigroups of Ladyzhenskaya type are squeezing semigroups with parameters $\Big(n, (1+\frac{1}{\alpha})\eta, 1+\alpha, \kappa\Big)$
and hence, quasi-stable with parameters $((1+\frac{1}{\alpha})\eta,\kappa(1+\alpha))$ by Proposition~\ref{prop:SQ2QUASIBANACH}. Indeed, let $\alpha>0$ be so large that $(1+\frac{1}{\alpha})\eta<1$ and suppose $x,y\in B$ are such that
$$\norm{S(T)x-S(T)y}_{X}>(1+\alpha)\norm{PS(T)x-PS(T)y}_{X}.$$
Then we get by (a)
\begin{equation}\label{e:INSTEADOFPYTHAGORAS}
\begin{split}
\norm{S(T)x-S(T)y}_{X}&\leq\norm{PS(T)x-PS(T)y}_{X}+\norm{(I-P)S(T)x-(I-P)S(T)y}_{X}\\
&<\tfrac{1}{1+\alpha}\norm{S(T)-S(T)y}_{X}+\eta\norm{x-y}_{X}
\end{split}
\end{equation}
and consequently,
$$\norm{S(T)x-S(T)y}_{X}<\left(1+\tfrac{1}{\alpha}\right)\eta\norm{x-y}_{X}.$$
Hence, if $B$ is a nonempty bounded subset of $V$ such that $S(T)B\subseteq B$, then Proposition~\ref{prop:QUASISQIMPROVEMENT} implies that, for any $\sigma\in(0,1-(1+\frac{1}{\alpha})\eta)$, the covering condition \eqref{e:CONDBALLSSEMB2} holds  with $q=(1+\frac{1}{\alpha})\eta+\sigma$ and 
$$h\leq\left(1+\tfrac{2\kappa(1+\alpha)d_{BM}(X_n,\ell^n_2)}{\sigma}\right)^{\mathbf{n}}$$
with $\mathbf{n}$ given in \eqref{e:ELL}. 

\item[(ii)] In the classical setting for Ladyzhenskaya type semigroups in \emph{Hilbert spaces}, $X_n$ is an $n$-dimensional subspace of a Hilbert space $X$ and $P$ is an \emph{orthogonal projection} of $X$ onto $X_n$. Thus instead of the triangle inequality in \eqref{e:INSTEADOFPYTHAGORAS} we can use the Pythagorean Theorem and conclude that Ladyzhenskaya type semigroups are squeezing semigroups with parameters $\Big(n, \sqrt{1+\frac{1}{\alpha^2}}\eta, \sqrt{1+\alpha^2}, \kappa\Big)$ provided that $\sqrt{1+\frac{1}{\alpha^2}}\eta<1$.
In this case, if $B$ is a~nonempty bounded subset of $V$ such that $S(T)B\subseteq B$, for any $\sigma\in\Big(0,1-\sqrt{1+\frac{1}{\alpha^2}}\eta\Big)$ the covering condition \eqref{e:CONDBALLSSEMB2} holds  with $q=\sqrt{1+\frac{1}{\alpha^2}}\eta+\sigma$ and 
$$h\leq\left(1+\tfrac{2\kappa\sqrt{1+\alpha^2}}{\sigma}\right)^{\mathbf{n}}.$$

\item[(iii)] Remaining in the Hilbert setting of (ii), we observe that condition (a) with $\mu=1$ in the definition of the squeezing property (Definition~\ref{defn:BANACHSQU}) implies property (a) in the definition of Ladyzhenskaya type semigroups (Definition~\ref{defn:LADY}).

Indeed, (a) in Definition~\ref{defn:BANACHSQU}  with $\mu=1$ is equivalent to the statement that for $x,y\in B$ either
\begin{equation}\label{e:SQ1zeta1}
\norm{S(T)x-S(T)y}_{X}\leq\norm{P(S(T)x-S(T)y)}_{X}
\end{equation}
holds or
\begin{equation}\label{e:SQ2zeta1}
\norm{S(T)x-S(T)y}_{X}\leq\eta\norm{x-y}_{X}.
\end{equation}
If \eqref{e:SQ1zeta1} holds, then $(I-P)(S(T)x-S(T)y)=0$ which implies \eqref{e:LADY}. On the other hand, if \eqref{e:SQ2zeta1} is satisfied then \eqref{e:LADY} certainly also holds, since $\norm{I-P}_{\mathcal{L}(X,X)}\leq 1$.
\end{itemize}
\end{rem}

We now show that semigroups of Ladyzhenskaya type satisfy both, the smoothing property and the generalized squeezing property with $\mu=1$ and hence, they are also quasi-stable from this point of view.

\begin{prop}\label{pro:LAD_SQ_SM}
Let $\{S(t)\colon t\geq0\}$ be a semigroup on a nonempty subset $V$ of a normed space $(X,\norm{\cdot}_{X})$ over $\K\in\{\R,\C\}$. If the semigroup is of Ladyzhenskaya type on a subset $B$ of $V$ at time $T>0$ with parameters $(n,\eta,\kappa)$ then it satisfies 
\begin{itemize}
\item[$(i)$] the smoothing property in Definition~\ref{defn:smooth} with \eqref{eq:smoothingRadek}, parameters $(\eta,\kappa)$ and
$$C(T)= (I-P)S(T),\quad M(T)=PS(T),\quad Z=X_n,$$ 
\item[$(ii)$] the generalized squeezing property in Definition~\ref{defn:BANACHSQU} with parameters $(n,\eta,1,\kappa)$. 
\end{itemize}
Moreover, if $B$ is a nonempty bounded subset of $V$ such that $S(T)B\subseteq B$ and $\sigma\in(0,1-\eta)$, the covering condition \eqref{e:CONDBALLSSEMB2} holds with  $q=\eta+\sigma$ and 
$$h\leq\left(1+\tfrac{2\kappa d_{BM}(X_n,\ell^n_2)}{\sigma}\right)^{\mathbf{n}},$$
where $\mathbf{n}$ is given in \eqref{e:ELL}. 

Consequently, if $\{S(t)\colon t\geq 0\}$ is a semigroup on a nonempty closed subset $V$ of a Banach space $X$ with a bounded absorbing set $\mathbf{B}\subseteq V$ that is  asymptotically closed or if $\mathbf{B}$ is closed, then for any $\sigma\in (0,1-\eta)$ there exists a $T$-discrete exponential attractor $\mathbf{M_0}$ for the semigroup and 
\begin{equation*}
\dimf^{V}(\mathbf{M_0})\leq \mathbf{n}\log_{\frac{1}{\eta+\sigma}}\left(1+\tfrac{2\kappa d_{BM}(X_n,\ell^n_2)}{\sigma}\right).
\end{equation*}
\end{prop}

\begin{proof}
Let $x,y\in B$. To verify the smoothing property, note that $M(T)=PS(T)\colon B\to X_n$ and $X_n$ is compactly embedded into $X$. The Lipschitz continuity of $PS(T)$ on $B$ yields
$$\norm{M(T)x-M(T)y}_{X_n}=\norm{PS(T)x-PS(T)y}_{X_n}= \norm{PS(T)x-PS(T)y}_{X}\leq \kappa\norm{x-y}_{X},$$
which shows that $M(T)$ satisfies \eqref{eq:smoothingRadek} with $Z=X_n$. Moreover, by assumption $C(T)$ is a~contraction on $B$ in $X$.

To show the generalized squeezing property we observe that 
\begin{align}\label{eq:Lad_gen_squ}
\begin{split}
\norm{S(T)x-S(T)y}_{X}&\leq\norm{(I-P)S(T)x-(I-P)S(T)y}_{X}+\norm{PS(T)x-PS(T)y}_{X}\\
&\leq \eta\norm{x-y}_{X}+\norm{PS(T)x-PS(T)y}_{X_n}
\end{split}
\end{align}
and hence, %the generalized squeezing property 
\eqref{e:BANACHSQU2} holds with $\mu=1$.
The statement now follows from Propositions \ref{prop:SQ2QUASIBANACH} and \ref{prop:QUASISQIMPROVEMENT} by taking $\mu=1$. The final claim is a consequence of Theorem~\ref{thm:SQEXPIMPROVED}.
\end{proof}

We observe that a semigroup satisfying the smoothing property from Definition 5.1 with (5.2) in a \emph{Hilbert} space is of Ladyzhenskaya type; hence it is also a squeezing semigroup by Remark 7.2 (ii).  

\begin{prop}
If a semigroup $\{S(t)\colon t\geq 0\}$ on a nonempty subset $V$ of a Hilbert space $(X,\|\cdot\|_{X})$ satisfies the smoothing property 
\eqref{eq:contraction}, \eqref{eq:smoothingRadek} on a subset $B$ of $V$ at time $T>0$ with parameters $(\eta,\kappa)$, 
then the semigroup is of Ladyzhenskaya type on $B$ at time $T$.
\end{prop}

\begin{proof}
Let $\eps>0$ be so small that $\eta+\eps\kappa<1$. Since the unit ball $\overline{B}^{Z}(0,1)$ in the normed space $Z$ is precompact in $X$, we have
$$\overline{B}^{Z}(0,1)\subseteq\bigcup_{i=1}^{p}B^{X}(x_i,\eps)$$
for some $x_i\in\overline{B}^{Z}(0,1)$. We consider
$$X_n=\lin\{x_1,\ldots,x_p\}\subseteq Z\subseteq X$$
and an orthogonal projection $P_n\colon X\to X_n$, where $n$ indicates the dimension of the finite-dimensional space $X_n$. Then we have
$$\norm{(I-P_n)z}_{X}=\inf_{x\in X_n}\norm{z-x}_{X}<\eps,\ z\in \overline{B}^{Z}(0,1),$$
which implies that
\begin{equation}\label{e:SMALLIMINUSP}
\norm{(I-P_n)z}_{X}\leq\eps\norm{z}_{Z},\ z\in Z.
\end{equation} 
By \eqref{eq:contraction} and \eqref{eq:smoothingRadek} we get for $x,y\in B$
\begin{equation*}
\begin{split}
\norm{P_n(S(T)x-S(T)y)}_{X}&\leq\norm{S(T)x-S(T)y}_{X}\leq\norm{C(T)x-C(T)y}_{X}\\
&+\norm{M(T)x-M(T)y}_{X}\leq(\eta+c_{Z,X}\kappa)\norm{x-y}_{X}
\end{split}
\end{equation*}
with the embedding constant $c_{Z,X}$, whereas by \eqref{eq:contraction}, \eqref{eq:smoothingRadek} and \eqref{e:SMALLIMINUSP} we obtain
\begin{equation*}
\begin{split}
\norm{(I-P_n)(S(T)x-S(T)y)}_{X}&\leq\norm{C(T)x-C(T)y}_{X}+\norm{(I-P_n)(M(T)x-M(T)y)}_{X}\\
&\leq\eta\norm{x-y}_{X}+\eps\norm{M(T)x-M(T)y}_{Z}\leq(\eta+\eps\kappa)\norm{x-y}_{X},
\end{split}
\end{equation*} 
which shows that the semigroup is of Ladyzhenskaya type on $B$ at time $T$ with parameters $(n,\eta+\eps\kappa,\eta+c_{Z,X}\kappa)$.
\end{proof}

Similarly as in the previous section for squeezing semigroups, if the phase space is a Hilbert space, we can even further improve the parameters in the covering condition in Proposition \ref{pro:LAD_SQ_SM} exploiting the property (a) in Definition~\ref{defn:LADY}. Note that we obtain a larger range for $\sigma$  and a~smaller value for $q$ in the following proposition. 

\begin{prop}\label{prop:LAD}
Let $\{S(t)\colon t\geq 0\}$ be a semigroup on a nonempty subset $V$ of a~Hilbert space $(X,\|\cdot\|_{X})$ over $\K\in\{\R,\C\}$, $T>0$ 
and let $B$ be a bounded subset of $V$ such that $S(T)B\subseteq B$.
If the semigroup is of Ladyzhenskaya type on $B$ at time $T$ with parameters $(n,\eta,\kappa)$ with an orthogonal projection $P\colon X\to X_n$ onto an $n$-dimensional subspace $X_n$ of $X$, then for any $\sigma \in (0,\sqrt{1-\eta^2})$
the covering condition \eqref{e:CONDBALLSSEMB2} holds with
$q=\sqrt{\sigma^2+\eta^2}$ and 
$$h=\left(1+\tfrac{2\kappa}{\sigma}\right)^{\mathbf{n}}$$
with $\mathbf{n}$ defined in \eqref{e:ELL}.
\end{prop}

\begin{proof}
We proceed exactly as in the proof of Proposition~\ref{prop:QUASISQIMPROVEMENT} treating $B$ as a precompact pseudometric space with the pseudometric
$$\rho(x,y)=\norm{P(S(T)x-S(T)y)}_{X},\ x,y\in B.$$
This implies that for any nonempty $A\subseteq B$ and $\sigma,\eps>0$ we have
\begin{equation}\label{e:ESTCRHO}
c_\rho(A,\eps,\sigma\eps)\leq \left(1+\tfrac{2\kappa}{\sigma}\right)^{\mathbf{n}}.
\end{equation} 
Instead of directly applying Lemma~\ref{lem:FUNDAMENTAL} we use the following refinement: If $\widehat{N}^{V}(A,\eps)<\infty$ for a subset $A\subseteq B$, then for any $\sigma>0$ it holds that
\begin{equation}\label{e:CRUCIALLADY}
\widehat{N}^{V}\Big(S(T)A,\sqrt{\sigma^2+\eta^2}\eps\Big)\leq \widehat{N}^{V}(A,\eps)c_\rho(A,\eps,\sigma\eps).
\end{equation}
This is a consequence of the proof of Lemma~\ref{lem:FUNDAMENTAL}, where in the last argument we 
use \eqref{e:LADY} and the Pythagorean Theorem  instead of \eqref{eq:Lad_gen_squ}. 
Indeed, for any $x,y\in C_{j}^{i}\subseteq F_i$ we have
\begin{align*}
\norm{S(T)x-S(T)y}_{X}^2&=\norm{P(S(T)x-S(T)y)}_{X}^{2}+\norm{(I-P)(S(T)x-S(T)y)}_{X}^{2}\\
&\leq 4\sigma^2\eps^2+4\eta^2\eps^2=4\eps^2(\sigma^2+\eta^2),
\end{align*}
and consequently, 
$$\diam^{V}(S(T)C_{j}^{i})\leq2\eps\sqrt{\sigma^2+\eta^2}.$$
The statement now follows from the proof of 
Theorem~\ref{thm:QUASI} using \eqref{e:ESTCRHO} and \eqref{e:CRUCIALLADY}.
\end{proof}

Finally, combining Proposition \ref{prop:LAD} with Theorem \ref{thm:EXDEXPAT} and Remark \ref{rem:closed} we obtain the following existence theorem for $T$-discrete exponential attractors.

\begin{thm}
Let $\{S(t)\colon t\geq 0\}$ be a semigroup on a nonempty closed subset $V$ of a Hilbert space $(X,\|\cdot\|_{X})$ over $\K\in\{\R,\C\}$ and $\mathbf{B}\subseteq V$ be a bounded absorbing set for the semigroup. Moreover, let $\{S(t)\colon t\geq 0\}$ be asymptotically closed or $\mathbf{B}$ be closed. 

If the semigroup is of Ladyzhenskaya type on $\mathbf{B}$ at time $T$ with parameters $(n,\eta,\kappa)$ with an orthogonal projection $P$ onto an $n$-dimensional subspace $X_n$ of $X$, then for any  $\sigma \in (0,\sqrt{1-\eta^2})$ there exists a~$T$-discrete exponential attractor  $\mathbf{M_0}\subseteq \mathbf{B}$ for the semigroup and
\begin{equation*}
\dimf^{V}(\mathbf{M_0})\leq \mathbf{n}\log_{\tfrac{1}{ \sqrt{\sigma^2+\eta^2}}}\left(1+\tfrac{2\kappa}{\sigma}\right).
\end{equation*}
If the semigroup is asymptotically closed, then it has a global attractor $\mathbf{A}$  contained in $\mathbf{M_0}$.
\end{thm}

\section{Construction for $C^1$ semigroups with global attractors}\label{sec:C1}

We present another method to construct $T$-discrete exponential attractors, which was developed by Y.~S.~Zhong, C.~K.~Zhong in \cite{ZZ12} and is based on an earlier paper by L.~Dung, B.~Nicolaenko \cite{DN01}, see also \cite{CLR10, CCLR22}. It requires the existence of a global attractor $\mathbf{A}$, the continuous differentiability of the map $S(T)$ in a neighborhood of $\mathbf{A}$ and a special structure of the derivatives of $S(T)$. More precisely, we assume the following:  

\begin{assc1}
Let $\{S(t)\colon t\geq 0\}$ be a semigroup on a normed space $(X,\norm{\cdot})$ over $\K\in\{\R,\C\}$ with a global attractor $\mathbf{A}$. 
We assume that for some $T>0$ the map $S(T)$ is $C^1$ on a~$\delta_0$-neighborhood of $\mathbf{A}$, 
$$B_{\delta_0}(\mathbf{A})=\bigcup_{x\in\mathbf{A}}B^{X}(x,\delta_0)$$
with some $\delta_0>0$, and there exists $\lambda\in(0,\frac{1}{4})$ such that for any $y\in B_{\delta_0}(\mathbf{A})$  the derivative $D_yS(T)$ decomposes as 
\begin{equation}\label{e:SMALLLAM}
D_{y}S(T)=K_{y}+C_{y},\ K_{y}\in\mathcal{L}(X)\text{ is compact},\ C_{y}\in\mathcal{L}(X),\ \norm{C_y}<\lambda.
\end{equation}
\end{assc1}

We start with a basic observation from \cite[Lemma 2.4]{CLR10}.

\begin{lem}\label{lem:APPROXFINITE}
Let $K\in\mathcal{L}(X)$ be a compact operator and $C\in\mathcal{L}(X)$ be a bounded operator in a normed space $(X,\norm{\cdot})$.
For any $\mu>\norm{C}$ there exists $m\in\N_0$ and an $m$-dimensional subspace $F$ of $X$ such that
\begin{equation*}
\sup_{w\in X, \norm{w}\leq 1}\inf_{z\in F, \norm{z}\leq 1}\norm{(K+C)(w-z)}<2\mu. 
\end{equation*}
\end{lem} 

\begin{proof}
Suppose contrary to the claim that for some $\mu>\norm{C}$, any $m\in\N_0$ and any $m$-dimensional subspace $F$ of $X$ we have
\begin{equation}\label{e:SMALLLAM2CONTRA}
\sup_{w\in X, \norm{w}\leq 1}\inf_{z\in F, \norm{z}\leq 1}\norm{(K+C)(w-z)}\geq 2\mu. 
\end{equation}
Let $x_1\in X$ be such that $\norm{x_1}\leq 1$ and set $F=\lin\{x_1\}$. Then $\dim{F}\leq 1$ and by \eqref{e:SMALLLAM2CONTRA} there exists $x_2\in X$, $\norm{x_2}\leq 1$ such that $\norm{(K+C)(x_2-x_1)}>\mu+\norm{C}$. Taking $F=\lin\{x_1,x_2\}$ we have $\dim{F}\leq 2$ and again by \eqref{e:SMALLLAM2CONTRA} there exists $x_3\in X$, $\norm{x_3}\leq 1$ such that $\norm{(K+C)(x_3-x_i)}>\mu+\norm{C}$ for $i=1,2$. By induction, there exists a~sequence $x_j\in X$, $\norm{x_j}\leq 1$ such that $\norm{(K+C)(x_j-x_l)}>\mu+\norm{C}$ for $j\neq l$.
Thus we get
$$\mu+\norm{C}<\norm{K x_j-K x_l}+\norm{C}\norm{x_j-x_l}\leq \norm{K x_j-K x_l}+2\norm{C},\ j\neq l,$$
and we have $\norm{K x_j-K x_l}>\mu-\norm{C}>0$ for $j\neq l$.
Since the sequence $K x_j$ contains a~Cauchy subsequence, we get a~contradiction.
\end{proof}

Let Assumption~$C^1$ hold.
Since by \eqref{e:SMALLLAM} $\lambda>\norm{C_y}$ for all $y\in B_{\delta_0}(\mathbf{A})$, Lemma~\ref{lem:APPROXFINITE} implies that
for any $y\in B_{\delta_0}(\mathbf{A})$ there exists $m\in\N_0$ and an $m$-dimensional subspace $F_{y}$ of $X$ such that
\begin{equation}\label{e:SMALLLAM2}
\sup_{w\in X, \norm{w}\leq 1}\inf_{z\in F_y, \norm{z}\leq 1}\norm{D_y S(T)(w-z)}<2\lambda. 
\end{equation}
Therefore the following number is well-defined: 
\begin{equation*}
\nu_{\lambda}(D_y S(T))=\min\{m\in\N_0\colon\dim{F_y}=m, F_y\text{ is a subspace of }X\text{ satisfying }\eqref{e:SMALLLAM2}\}.  
\end{equation*}

Next we estimate the number of open balls of smaller radius needed to cover a given closed ball in a finite-dimensional space using the Banach-Mazur distance. 

\begin{lem}\label{lem:FINITE}
Let $X_n$ be an $n$-dimensional subspace of a normed space $X$ over $\K\in\{\R,\C\}$. Then for $0<\eps<r$ we have
\begin{equation}\label{e:MINIMALCOVERFINITE}
\overline{B}^{X_n}(0,r)\subseteq\bigcup_{i=1}^{h}B^{X_n}(y_i,\eps)
\end{equation}
for some $y_1,\ldots,y_h\in X_n$, where
\begin{equation}\label{e:MINIMALCOVERFINITEESTIMATE}
h\leq\Bigl(1+\frac{2d_{BM}(X_n,\ell^{n}_2)r}{\eps}\Bigr)^{\mathbf{n}}\leq\Bigl(1+\frac{2\sqrt{n}r}{\eps}\Bigr)^{\mathbf{n}}
\end{equation}
with $\mathbf{n}$ given in \eqref{e:ELL}.
\end{lem}

\begin{proof}
If $n=0$ then $X_n=\{0\}$ and the result holds trivially with $h=1$, so let $n\in\N$. 
Let $h\in\N$ be the minimal number such that \eqref{e:MINIMALCOVERFINITE} holds with some $y_1,\ldots, y_h\in X_n$.
We consider an arbitrary isomorphism $\Lambda\colon X_n\to\ell^{n}_2$ and the (real) isometry $R\colon\ell^{n}_2\to\R^{\mathbf{n}}$. We have
$$\overline{B}^{X_n}(0,r)=\Lambda^{-1}\Lambda \overline{B}^{X_n}(0,r)\subseteq\Lambda^{-1}R^{-1}\overline{B}^{\R^{\mathbf{n}}}(0,\|\Lambda\|_{\mathcal{L}(X_n,\ell^n_2)}r).$$
Let
$$m=m_{|\cdot|_2}\Big(\overline{B}^{\R^{\mathbf{n}}}(0,\|\Lambda\|_{\mathcal{L}(X_n,\ell^n_2)}r),\frac{\eps}{\|\Lambda^{-1}\|_{\mathcal{L}(\ell^n_2,X_n)}}\Big).$$
Then, as in \eqref{e:BOUNDOFM}, we get
$$m\leq\Bigl(1+\frac{2\|\Lambda\|_{\mathcal{L}(X_n,\ell^n_2)}\|\Lambda^{-1}\|_{\mathcal{L}(\ell^n_2,X_n)}r}{\eps}\Bigr)^{\mathbf{n}}$$
and for some $x_1,\ldots,x_m\in \overline{B}^{\R^{\mathbf{n}}}(0,\|\Lambda\|_{\mathcal{L}(X_n,\ell^n_2)}r)$ we have
$$\overline{B}^{\R^{\mathbf{n}}}(0,\|\Lambda\|_{\mathcal{L}(X_n,\ell^n_2)}r)\subseteq\bigcup_{i=1}^{m}B^{\R^{\mathbf{n}}}(x_i,\frac{\eps}{\|\Lambda^{-1}\|_{\mathcal{L}(\ell^n_2,X_n)}}).$$
Consequently, for some $y_1,\ldots,y_m\in X_n$ we get
$$\overline{B}^{X_n}(0,r)\subseteq\bigcup_{i=1}^{m}B^{X_n}(y_i,\eps).$$
Thus $h\leq m$ and taking the infimum over all isomorphisms $\Lambda$, we obtain \eqref{e:MINIMALCOVERFINITEESTIMATE} with the help of F.~John's Theorem, see \eqref{e:JOHN}.   
\end{proof}

\begin{thm}\label{thm:C1SEMIGROUP}
Let Assumption~$C^1$ hold with $T>0$, $\delta_0>0$ and $\lambda\in(0,\frac{1}{4})$.
For any $\sigma\in(0,1-4\lambda)$ there exist $r_0>0$, $b_0>0$ and a~bounded absorbing set $\mathbf{B}\subset B_{\delta_0}(\mathbf{A})$ such that $S(T)\mathbf{B}\subseteq\mathbf{B}$ and the covering condition \eqref{e:CONDBALLSSEMB} holds,
\begin{equation}\label{e:COVERINGC1}
N^{X}(S(kT)\mathbf{B},q^{k}r_0)\leq b_0h^{k},\ k\in\N,
\end{equation}
with $q=\sigma+4\lambda\in(0,1)$ and
\begin{equation}\label{e:ESTIMATEH}
h\leq\Bigl(1+\frac{8\sqrt{n}M}{\sigma}\Bigr)^{\mathbf{n}},
\end{equation}
where
$$n=\sup_{y\in\mathbf{B}}\nu_{\lambda}(D_y S(T)),\quad M=\sup_{y\in\mathbf{B}}\norm{D_y S(T)}$$
and $\mathbf{n}$ is given in \eqref{e:ELL}.
\end{thm}

\begin{proof}
\textbf{Step 1.} 
Let $\sigma\in(0,1-4\lambda)$ and $q=\sigma+4\lambda\in(0,1)$. By the continuity of $x\mapsto D_xS(T)$ on $\mathbf{A}$, for any $x\in\mathbf{A}$ there exists $0<r_x\leq\delta_0$ such that
\begin{equation}\label{e:CONTDER}
\norm{D_yS(T)-D_xS(T)}<\min\Big\{1,\frac{\lambda-\norm{C_x}}{2},\frac{\sigma}{8}\Big\}\text{ for }y\in B^{X}(x,r_x).
\end{equation}
By the compactness of $\mathbf{A}$ there exist $x_1,\ldots,x_p\in\mathbf{A}$ such that 
$$\mathbf{A}\subseteq\bigcup_{i=1}^{p}B^{X}(x_i,\tfrac{r_{x_i}}{2}).$$
We set $r_0=\frac{1}{2}\min\{r_{x_1},\ldots,r_{x_p}\}$ and find $0<\delta<\min\{\delta_0,qr_0\}$ such that
\begin{equation}\label{e:DELTADEF}
B_{\delta}(\mathbf{A})\subseteq\bigcup_{i=1}^{p}B^{X}(x_i,\tfrac{r_{x_i}}{2}),
\end{equation}
cf. \cite[Lemma 2.3]{ZZ12}. Therefore, we obtain
\begin{equation*}
\norm{D_yS(T)}\leq\norm{D_yS(T)-D_{x_i}S(T)}+\norm{D_{x_i}S(T)}\leq\sup_{x\in\mathbf{A}}\norm{D_xS(T)}+1<\infty,\ y\in B_\delta(\mathbf{A}).
\end{equation*}
\textbf{Step 2.} We show that
\begin{equation}\label{e:COMMONDIMENSION}
\sup\limits_{y\in B_{\delta}(\mathbf{A})}\nu_{\lambda}(D_y S(T))<\infty.
\end{equation}
Indeed, let $y\in B_{\delta}(\mathbf{A})$. By \eqref{e:CONTDER} and \eqref{e:DELTADEF} there exists $x_i\in\mathbf{A}$, such that
$$\norm{D_yS(T)-D_{x_i}S(T)}<\frac{\lambda-\|C_{x_i}\|}{2}.$$
Applying Lemma~\ref{lem:APPROXFINITE} with $x_i\in\mathbf{A}$ and $\mu_{x_i}=\frac{\lambda+\|C_{x_i}\|}{2}$, 
there exists a~finite-dimensional subspace $F_{x_i}$ of $X$ such that
\begin{equation*}
\sup_{w\in X, \norm{w}\leq 1}\inf_{z\in F_{x_i}, \norm{z}\leq 1}\norm{D_{x_i} S(T)(w-z)}<\lambda+\norm{C_{x_i}}. 
\end{equation*}
Thus for any $w\in X$, $\norm{w}\leq 1$ we find $z\in F_{x_i}$, $\norm{z}\leq 1$ such that
\begin{equation*}
\begin{split}
\norm{D_y S(T)(w-z)}&\leq\norm{(D_y S(T) -D_{x_i} S(T)) w}+\norm{D_{x_i} S(T) (w-z)}\\
&+\norm{(D_{x_i} S(T)-D_yS(T))z}<2\lambda.
\end{split}
\end{equation*}
This shows that 
$$\sup_{y\in B_{\delta}(\mathbf{A})}\nu_{\lambda}(D_y S(T))\leq\max\{\dim F_{x_1},\ldots,\dim F_{x_p}\}<\infty,$$
which proves \eqref{e:COMMONDIMENSION}.

\textbf{Step 3.} 
Let  $y\in B_{\delta}(\mathbf{A})$ and $z\in B^{X}(y,r_0)$. 
Then there exists $x_i\in\mathbf{A}$ such that $y\in B^{X}(x_i,\frac{r_{x_i}}{2})$. Thus $y+\tau(z-y)\in B^{X}(x_i,r_{x_i})\subseteq B_{\delta}(\mathbf{A})$ for $\tau\in[0,1]$ and by the fact that $S(T)$ is $C^1$ on $B_{\delta_0}(\mathbf{A})$ and by \eqref{e:CONTDER} we have
\begin{equation*}
\begin{split}
&\norm{S(T)z-S(T)y-D_y S(T)(z-y)}=\norm{\int_{0}^{1}(D_{y+\tau(z-y)}S(T)(z-y)-D_y S(T)(z-y))d\tau}\\
&\leq\int_0^1\norm{D_{y+\tau(z-y)}S(T)-D_{x_i}S(T)+D_{x_i}S(T)-D_y S(T)}d\tau\norm{z-y}\leq\frac{\sigma}{4}\norm{z-y}.
\end{split}
\end{equation*}

Since $B_{\delta}(\mathbf{A})$ is a bounded absorbing set, there exists $k_0\in\N$ such that
 $S(kT)B_\delta(\mathbf{A})\subseteq B_\delta(\mathbf{A})$ for $k\geq k_0$. Then the set
$$\mathbf{B}=\bigcup_{k\geq k_0}S(kT)B_\delta(\mathbf{A})\subseteq B_\delta(\mathbf{A})$$
is a bounded absorbing set satisfying $S(T)\mathbf{B}\subseteq\mathbf{B}$.

We know that $n=\sup\limits_{y\in \mathbf{B}}\nu_{\lambda}(D_y S(T))<\infty$, $M=\sup\limits_{y\in \mathbf{B}}\norm{D_{y} S(T)}<\infty$ and
\begin{equation}\label{e:THETAR2}
\norm{S(T)z-S(T)y-D_{y}S(T)(z-y)}\leq\frac{\sigma}{4}\norm{z-y},\ y\in \mathbf{B},\ z\in B^{X}(y,r_0). 
\end{equation}
\textbf{Step 4.} 
Let $0<r\leq r_0$ and $y\in\mathbf{B}$. First consider the case $D_y S(T)\neq 0$. Using the definition of $\nu_{\lambda}(D_y S(T))$ and \eqref{e:SMALLLAM2} there exists a finite-dimensional subspace $F_y$ of $X$ such that
$\dim F_y=n_y=\nu_{\lambda}(D_y S(T))$ and, given $w\in \overline{B}^{X}(0,r)$, there is $z\in\overline{B}^{F_y}(0,r)$ such that
$$\norm{D_y S(T)(w-z)}<2\lambda r.$$
Moreover, by Lemma~\ref{lem:FINITE} there exist points $z_1,\ldots,z_h\in F_y$ with 
\begin{equation}\label{e:ESTH}
h\leq\Bigl(1+\frac{8\sqrt{n_y}\norm{D_y S(T)}}{\sigma}\Bigr)^{\mathbf{n}_y}\leq \Bigl(1+\frac{8\sqrt{n}M}{\sigma}\Bigr)^{\mathbf{n}}
\end{equation}
such that for some $z_i$ we have $\norm{z-z_i}\leq\frac{\sigma r}{4\norm{D_y S(T)}}$. Consequently, we get
$$\norm{D_y S(T)(w-z_i)}<\frac{\sigma r}{4}+2\lambda r,$$
so
\begin{equation}\label{e:FIRSTCOVER}
D_y S(T)(\overline{B}^{X}(0,r))\subseteq\bigcup_{i=1}^{h}B^{X}(D_y S(T)z_i, \frac{\sigma r}{4}+2\lambda r)
\end{equation} 
with $h$ bounded as in \eqref{e:ESTH}.
Note that if $D_y S(T)=0$, then \eqref{e:FIRSTCOVER} holds trivially with $h=1$.

\textbf{Step 5.} 
Let now $0<r\leq r_0$, $y\in\mathbf{B}$ and $z\in B^{X}(y,r)$. Then by \eqref{e:FIRSTCOVER} there exists $z_i\in F_y$ such that
$$\norm{D_y S(T)(z-y)-D_y S(T) z_i}<\frac{\sigma r}{4}+2\lambda r.$$
Hence from \eqref{e:THETAR2} we get
$$\norm{S(T)z-S(T)y-D_y S(T) z_i}<\frac{\sigma r}{2}+2\lambda r,$$
and hence, with $q=\sigma+4\lambda$, it follows that
$$S(T)(B^{X}(y,r))\subseteq\bigcup_{i=1}^{h}B^{X}(S(T)y+D_y S(T) z_i,\frac{qr}{2}).$$
Thus for any nonempty subset $B$ of $\mathbf{B}$ there exist $y_1,\ldots,y_h\in S(T)(B^{X}(y,r)\cap B)$ such that
$$S(T)(B^{X}(y,r)\cap B)\subseteq\bigcup_{i=1}^{h}B^{X}(y_i,qr)$$
and consequently, for any nonempty subset $B$ of $\mathbf{B}$ and any $0<r\leq r_0$ we have
\begin{equation}\label{e:ITERATION}
N^{X}(S(T)(B^{X}(y,r)\cap B),q r)\leq h,\ y\in\mathbf{B},
\end{equation}
with $h$ estimated in \eqref{e:ESTIMATEH}.

Since $\delta<qr_0$, there are $a_1,\ldots,a_{N_0}\in\mathbf{A}\subseteq S(T)\mathbf{B}$ such that $\mathbf{A}\subseteq\bigcup\limits_{i=1}^{N_0}B^{X}(a_i,\frac{qr_0-\delta}{2})$,
which implies that
$$S(T)\mathbf{B}\subseteq \mathbf{B}\subseteq B_{\delta}(\mathbf{A})\subseteq\bigcup_{i=1}^{N_0}B^{X}(a_i,qr_0).$$
Therefore, we get by \eqref{e:ITERATION}
$$S(2T)\mathbf{B}\subseteq\bigcup_{i=1}^{N_0}S(T)(B^{X}(a_i,qr_0)\cap S(T)\mathbf{B})\subseteq\bigcup_{i=1}^{N_0}\bigcup_{j=1}^{h}B^{X}(a_{ij},q^2r_0)$$
with some $a_{ij}\in S(2T)\mathbf{B}$.
Iterating the argument and using repeatedly \eqref{e:ITERATION}, we obtain
$$N^{X}(S(kT)\mathbf{B},q^{k}r_0)\leq N_0h^{k-1},\ k\in\N,$$
which proves \eqref{e:COVERINGC1}.
\end{proof} 

Combining Theorem~\ref{thm:C1SEMIGROUP} with Corollary~\ref{cor:WITHGLOBALATTR}, we obtain the following.

\begin{cor}
Let Assumption~$C^1$ hold with $T>0$, $\delta_0>0$ and $\lambda\in(0,\frac{1}{4})$.
Then for any $\sigma\in(0,1-4\lambda)$ there exists a bounded absorbing set $\mathbf{B}\subset B_{\delta_0}(\mathbf{A})$ satisfying $S(T)\mathbf{B}\subseteq\mathbf{B}$ and a certain countable subset $\mathbf{E_0}$ of $\mathbf{B}$ such that $\mathbf{M_0}=\mathbf{A}\cup\mathbf{E_0}=\cl_X\mathbf{E_0}\subseteq\mathbf{B}$ is a $T$-discrete exponential attractor in $X$
with rate of attraction $\xi\in(0,\frac{1}{T}\ln{\frac{1}{\sigma+4\lambda}})$, and its fractal dimension is bounded by
$$\dimf^{X}(\mathbf{M_0})\leq\mathbf{n}\log_{\frac{1}{\sigma+4\lambda}}\Bigl(1+\frac{8\sqrt{n}M}{\sigma}\Bigr),$$
where $\mathbf{n}$ is given in \eqref{e:ELL} with
$$n=\max_{y\in\mathbf{B}}\nu_{\lambda}(D_y S(T))\text{ and }M=\sup_{y\in\mathbf{B}}\norm{D_y S(T)}.$$
\end{cor}

\begin{rem}
Following \cite[Section 3]{DN01}, if  $S(T)B_{\delta_0}(\mathbf{A})\subseteq B_{\delta_0}(\mathbf{A})$ and \eqref{e:SMALLLAM} is replaced by
\begin{equation*}
D_{y}S(T)=K_{y}+C_{y},\ K_{y}\in\mathcal{L}(X)\text{ is compact},\ C_{y}\in\mathcal{L}(X),\ \norm{C_y}<\lambda_0<1
\end{equation*}
for any $y\in B_{\delta_0}(\mathbf{A})$, then
$$D_y(S(kT))=D_{S(T)^{k-1}y}S(T)\circ\ldots\circ D_yS(T),\ y\in B_{\delta_0}(\mathbf{A}),\ k\in\N.$$
For example, for $k=2$ we have
$$D_y(S(2T))=K_{S(T)y}\circ K_y+C_{S(T)y}\circ K_y+C_{S(T)y}\circ C_y,$$
where $K_{S(T)y}\circ K_y$ and $C_{S(T)y}\circ K_y$ are compact and $\norm{C_{S(T)y}\circ C_y}<\lambda_0^2$. 
Therefore, by induction we have
$$D_y(S(kT))=\tilde{K}_y+\tilde{C}_y,\ y\in B_{\delta_0}(\mathbf{A}),\ k\in\N,$$
where $\tilde{K}_y\in\mathcal{L}(X)$ is compact and $\tilde{C}_y\in\mathcal{L}(X)$ with $\|\tilde{C}_y\|<\lambda_0^k$.
Taking $k_0\in\N$ such that $\lambda=\lambda_0^{k_0}\in(0,\frac{1}{4})$, we can apply Theorem~\ref{thm:C1SEMIGROUP} to $S(k_0T)$ in the role of $S(T)$ and obtain the existence of a $k_0T$-discrete exponential attractor.
\end{rem} 

\section{Existence results for classical exponential attractors}\label{sec:EA}

We now turn to the existence of exponential attractors in the classical sense when a~semigroup is defined in the time interval $[0,\infty)$. To this end, we need an additional property that allows us to extend the $T$-discrete exponential attractor to a compact set of finite fractal dimension that is positively invariant for all $t\in[0,\infty)$. A sufficient condition is the H\"older continuity in time of the semigroup which is a restrictive assumption. 
Following~\cite{sonner2012} we obtain the following result. Note that we obtain a better estimate for the fractal dimension of the exponential attractor than in \cite{chueshov, CzEf}.

\begin{thm}\label{thm:SEM3}
Let $\{S(t)\colon t\geq 0\}$ be an asymptotically closed semigroup on a~complete metric space $(V,d)$, $\mathbf{B}$ be a bounded absorbing set, $T>0$ and assume
that the covering condition \eqref{e:CONDBALLSSEMB}
holds with some $k_0\in\N$, $q\in(0,1)$, $a,b>0$ and $h\geq 1$.
Assume further that there exist $T_2>T_1\geq 0$,
$\zeta>0$ and $\nu>0$ such that
\begin{equation}\label{e:TIMEHOLDER}
d(S(t_1)x,S(t_2)x)\leq\zeta\abs{t_1-t_2}^{\nu},\ t_1,t_2\in[T_1,T_2],\ x\in\mathbf{B}.
\end{equation}
If $T_1>0$ then we also assume that for some $N\in\N$ such that $NT\geq T_1$
\begin{equation}\label{e:LIPTSEMN}
d(S(NT)x,S(NT)y)\leq L_{N}d(x,y),\ x,y\in \mathbf{B},
\end{equation}
holds with some $L_{N}\geq 0$. Then there exists an exponential attractor
$\mathbf{M}$ (independent of $\nu$, $\zeta$, $T_2$) with rate of attraction $\xi\in(0,\frac{1}{T}\ln\frac{1}{q})$
and its fractal dimension is bounded by
\begin{equation*}
\dimf^{V}(\mathbf{M})\leq\tfrac{1}{\nu}+\log_{\frac{1}{q}}h.
\end{equation*}
Moreover,
\begin{equation*}
\mathbf{M}=\mathbf{A}\cup\mathbf{E}=\cl_{V}\mathbf{E}\subseteq\mathbf{B},
\end{equation*}
where $\mathbf{E}=\bigcup_{t\in I}S(t)\mathbf{E_0} \subseteq\mathbf{B}$ for some compact interval $I$, $\mathbf{E_0}$ is the set constructed in Theorem~\ref{thm:EXDEXPAT} and $\mathbf{A}=\Lambda^{V}(\cl_{V}\mathbf{E_0})$ is the global attractor for the semigroup.
\end{thm}

\begin{proof}
\textbf{Step 1.} Without loss of generality we can assume that $\mathbf{B}$ is positively invariant by Remark~\ref{rem:POSINV}. We first show that there exists a~set $\mathbf{E}$ (independent of $\nu$, $\zeta$, $T_2$) with the following properties:
\begin{itemize}
\item[(i)] $\mathbf{E}\subseteq\mathbf{B}$, $\mathbf{E}$ is precompact in $V$,
\item[(ii)] $S(t)\mathbf{E}\subseteq\mathbf{E},\ t\geq 0,$
\item[(iii)] the fractal dimension of $\mathbf{E}$ is bounded by
\begin{equation*}
\dimf^{V}(\mathbf{E})\leq\tfrac{1}{\nu}+\log_{\frac{1}{q}}h,
\end{equation*}
\item[(iv)] for any $\xi\in(0,\frac{1}{T}\ln\frac{1}{q})$,
and any bounded subset $G$ of $V$ we have
\begin{equation}\label{e:EXPATESEM2}
\lim_{t\rightarrow\infty}e^{\xi t}\dist^{V}(S(t)G,\mathbf{E})=0.
\end{equation}
\end{itemize}
Indeed, if $T_1=0$ then we set $N=0$, otherwise let $N\in\N$ be such that $NT\geq T_1$
and \eqref{e:LIPTSEMN} holds. We define
$$\mathbf{E}=\bigcup_{p\in[NT,(N+1)T]}S(p)\mathbf{E_0},$$
where $\mathbf{E_0}$, constructed in Step 2 in the proof of Theorem~\ref{thm:EXDEXPAT}, satisfies ($e_1$), ($e_2$) from there. Since $\mathbf{E_0}\subseteq \mathbf{B}$
and $\mathbf{B}$ is positively invariant, we get $\mathbf{E}\subseteq \mathbf{B}$. For $t\geq 0$ we observe that
\begin{align*}
S(t)\mathbf{E}&=\bigcup_{p\in[NT,(N+1)T]}S(p+t)\mathbf{E_0}\subseteq\bigcup_{s\in[0,T)}\bigcup_{l\in\N_0}S((N+l)T+s)\mathbf{E_0}\\
&\subseteq\bigcup_{s\in[0,T)}S(NT+s)\mathbf{E_0}\subseteq\mathbf{E},
\end{align*}
which shows (ii).
Let $\xi\in(0,\frac{1}{T}\ln\frac{1}{q})$ and $G\subseteq V$ be bounded.
If $T_1=0$ then $\mathbf{E_0}\subseteq\mathbf{E}$ and \eqref{e:EXPATESEM2} follows directly from
($e_2$). Otherwise, \eqref{e:LIPTSEMN} is assumed and we
know that $S(t)G\subseteq\mathbf{B}$ for $t\geq t_{G}$ and for any $\eps>0$ there exists $t_\eps\geq 0$
such that
$$e^{\xi t}\dist^{V}(S(t)G,\mathbf{E_0})<\frac{\eps}{e^{\xi NT}(L_{N}+1)},\ t\geq t_\eps.$$
Fix $\eps>0$ and let $t\geq NT+t_\eps+t_{G}$. Then we have by \eqref{e:LIPTSEMN}
\begin{align*}
e^{\xi t}\dist^{V}(S(t)G,\mathbf{E})&\leq e^{\xi t}\dist^{V}(S(NT)S(t-NT)G,S(NT)\mathbf{E_0})\\
&\leq e^{\xi NT}L_{N}e^{\xi(t-NT)}\dist^{V}(S(t-NT)G,\mathbf{E_0})<\eps.
\end{align*}
We are left to prove (iii). Now let $\tau=T_2-T_1>0$. First we show that
\begin{equation}\label{e:HOLDERONSMALL}
d(S(t_1)x,S(t_2)x)\leq\zeta\abs{t_1-t_2}^{\nu},\ x\in \mathbf{B},\ t_1,t_2\in[T_1+l\tau,T_1+(l+1)\tau],\ l\in\N_0.
\end{equation}
Indeed, let $l\in\N_0$, $t_1,t_2\in[T_1+l\tau,T_1+(l+1)\tau]$ and $x\in\mathbf{B}$. Then $t_i=T_1+l\tau+s_i$, $s_i\in[0,\tau]$, $i=1,2$, and
$$d(S(t_1)x,S(t_2)x)=d(S(T_1+s_1)S(l\tau)x,S(T_1+s_2)S(l\tau)x)\leq\zeta\abs{t_1-t_2}^{\nu}.$$
By ($e_1$) from Step 2 in the proof of Theorem~\ref{thm:EXDEXPAT} we also have
$$\mathbf{E}=\bigcup_{k\geq k_0}\bigcup_{p\in[NT,(N+1)T]}S(p)Q_{k},$$
where $Q_k\subseteq S(kT)\mathbf{B}$ and $\displaystyle\# Q_k\leq b\sum_{j=0}^{k-k_0}h^{k-j}$.

Let $m_0\geq k_0$ be such that
\begin{equation}\label{e:M0}
aq^{m}\leq 2\zeta\tau^{\nu}\text{ for }m\geq m_0.
\end{equation}
Let $m\geq m_0$ and $p\in[NT,(N+1)T]$. For $k\geq m$ we have
$$S(p)Q_k\subseteq S(p)S(kT)\mathbf{B}=S(mT)S((k-m)T+p)\mathbf{B}\subseteq S(mT)\mathbf{B},$$
and thus
$$\mathbf{E}\subseteq\bigcup_{k=k_0}^{m}\bigcup_{p\in[NT,(N+1)T]}S(p)Q_k\cup S(mT)\mathbf{B},\ m\geq m_0.$$
We now construct a finite covering of this set
by balls with centers in $\mathbf{E}$ and radii $2aq^{m}$. 
To this end, we denote the elements of $\displaystyle\bigcup_{k=k_0}^{m}Q_k$ by $\{x_i\colon i=1,\ldots,i_m\}$ and observe that 
$$i_{m}=\# \bigcup_{k=k_0}^{m}Q_k\leq\sum_{k=k_0}^{m}\# Q_k\leq b\sum_{k=k_0}^{m}\sum_{j=0}^{k-k_0}h^{k-j}\leq b(m-k_0+1)^{2}h^{m}.$$
The interval $[NT,(N+1)T]$ can be covered by $\left[\frac{T}{\tau}\right]+2$ intervals of length $\tau$
of the form $[T_1+l\tau,T_1+(l+1)\tau]$ with $l\in\N_0$. 
We further subdivide each such interval into intervals of length
$\left(\frac{aq^{m}}{2\zeta}\right)^{\frac{1}{\nu}}\leq \tau$ 
by \eqref{e:M0}, plus possibly one interval of smaller length. We note that 
$\Big[\tau\Big(\frac{2\zeta}{aq^{m}}\Big)^{\frac{1}{\nu}}\Big]+1$ such intervals cover 
$[T_1+l\tau,T_1+(l+1)\tau]$, since
$\Big(\frac{aq^{m}}{2\zeta}\Big)^{\frac{1}{\nu}}\Big(\Big[\tau\Big(\frac{2\zeta}{aq^{m}}\Big)^{\frac{1}{\nu}}\Big]+1\Big)>\tau$.
Let $I_j$, $j=1,\ldots, j_m$, be an arbitrary one of these intervals, where
$$j_m\leq\left(\left[\tfrac{T}{\tau}\right]+2\right)\Big(\Big[\tau\Big(\tfrac{2\zeta}{aq^{m}}\Big)^{\frac{1}{\nu}}\Big]+1\Big).$$
We choose $p_j\in I_j$. Then, for any $p\in I_j$ we get by \eqref{e:HOLDERONSMALL}
$$d(S(p)x_i,S(p_j)x_i)\leq\zeta\abs{p-p_j}^{\nu}\leq\frac{1}{2}aq^{m}<aq^{m},$$
and thus, 
$$\bigcup_{i=1}^{i_m}\bigcup_{p\in[NT,(N+1)T]}S(p)x_i\subseteq\bigcup_{i=1}^{i_m}\bigcup_{j=1}^{j_m}B^{V}(S(p_j)x_i,aq^{m}).$$
This implies that
$$\bigcup_{k=k_0}^{m}\bigcup_{p\in[NT,(N+1)T]}S(p)Q_k\cap\mathbf{E}\subseteq\bigcup_{i=1}^{i_m}\bigcup_{j=1}^{j_m}B^{V}(u_{i,j},2aq^{m})$$
for some $u_{i,j}\in\mathbf{E}$. Using \eqref{e:CONDBALLSSEMB} and \eqref{e:M0} we obtain for $m\geq m_0$
\begin{equation*}
\begin{split}
N^{V}(\mathbf{E},2aq^{m})\leq i_m j_m+bh^{m}&\leq b(m-k_0+1)^{2}h^{m}\left(\left[\tfrac{T}{\tau}\right]+2\right)\left(\tau\left(\tfrac{2\zeta}{aq^{m}}\right)^{\frac{1}{\nu}}+1\right)+bh^{m}\\
&\leq c(m-k_0+1)^{2}h^{m}\left(\tfrac{2\zeta}{aq^{m}}\right)^{\frac{1}{\nu}},
\end{split}
\end{equation*}
where $c$ is a positive constant independent of $m$.
Since $2aq^m$ converges to $0$ as $m\to\infty$, it follows that $\mathbf{E}$
is precompact in $V$. Consider an arbitrary sequence $\eps_n>0$ converging to $0$ and choose
$m_n\in\N$ such that $m_n\geq m_0$ and
$$2aq^{m_n}\leq \eps_n<2aq^{m_n-1}<1\text{ for }n\geq n_0.$$
Then, we have
$$\log_{\frac{1}{\eps_n}}N^{V}(\mathbf{E},\eps_n)\leq\frac{\ln{c}+2\ln(m_n-k_0+1)+m_n\ln{h}+\frac{1}{\nu}\left(\ln(\frac{2\zeta}{a})-m_n\ln{q}\right)}{-\ln(2a)-(m_n-1)\ln{q}},$$
which implies that
$$\dimf^{V}(\mathbf{E})\leq\frac{\ln{h}-\tfrac{1}{\nu}\ln{q}}{-\ln{q}}=\tfrac{1}{\nu}+\log_{\frac{1}{q}}h$$
and completes the proof of the claims (i)--(iv).

\textbf{Step 2.}
By Theorem~\ref{thm:EXDEXPAT} the global attractor for the semigroup exists, $\mathbf{A}=\Lambda^{V}(\cl_{V}\mathbf{E_0})$  and
\begin{equation*}
\dimf^{V}(\mathbf{A})\leq\log_{\frac{1}{q}}h.
\end{equation*}
We define
$$\mathbf{M}=\mathbf{A}\cup\mathbf{E}\subseteq\mathbf{B}$$
and show that it is an exponential attractor. 

The set $\mathbf{M}$ is nonempty, positively invariant, precompact and
$$\dimf^{V}(\mathbf{M})=\max\{\dimf^{V}(\mathbf{A}),\dimf^{V}(\mathbf{E})\}\leq\tfrac{1}{\nu}+\log_{\frac{1}{q}}h.$$
Moreover,  for any $\xi\in(0,\frac{1}{T}\ln{\frac{1}{q}})$, and
any bounded subset $G$ of $V$ we have
$$\lim_{t\to\infty}e^{\xi t}\dist^{V}(S(t)G,\mathbf{M})=0.$$
Hence, it remains to show that $\mathbf{M}$ is compact.

Consider a sequence $x_n\in\mathbf{M}$, $n\in\N$. If infinitely many of its elements belong
to $\mathbf{A}$, then by the compactness of $\mathbf{A}$, there exists a subsequence convergent
to an element of $\mathbf{A}\subseteq\mathbf{M}$. Otherwise,  there is a subsequence
$$x_{n_j}\in\mathbf{E}=\bigcup_{k\geq k_0}\bigcup_{p\in[NT,(N+1)T]}S(p)Q_k.$$
Thus $x_{n_j}=S(p_j)y_j$, where $p_{j}\in[NT,(N+1)T]$, $y_j\in Q_{k_j}$ for some $k_j\geq k_0$.
Taking a~subsequence if necessary, we can assume that $p_j\to p_0\in[NT,(N+1)T]\subset[T_1,\infty)$
by the compactness of the interval.

If $p_0\in(T_1+l\tau,T_1+(l+1)\tau)$ for some $l\in\N_0$, then $p_j\in(T_1+l\tau,T_1+(l+1)\tau)$
for large $j$. Moreover, by \eqref{e:HOLDERONSMALL} we have
$$d(S(p_j)x,S(p_0)x)\leq\zeta\abs{p_j-p_0}^{\nu},\ x\in\mathbf{B}.$$
If $p_0=T_1+l\tau$ for some $l\in\N$, then $p_j\in(T_1+l\tau-\frac{\tau}{2},T_1+l\tau+\frac{\tau}{2})$
for large $j$. Moreover, for $x\in\mathbf{B}$ and $t_1,t_2\in(T_1+l\tau-\frac{\tau}{2},T_1+l\tau+\frac{\tau}{2})$
we have $t_i=T_1+l\tau-\frac{\tau}{2}+s_i$, $s_i\in(0,\tau)$, $i=1,2$ and by \eqref{e:TIMEHOLDER}
and the positive invariance of $\mathbf{B}$ we get
\begin{equation*}
\begin{split}
d(S(t_1)x,S(t_2)x)&=d(S(T_1+s_1)S(l\tau-\tfrac{\tau}{2})x,S(T_1+s_2)S(lT-\tfrac{\tau}{2})x)\\
&\leq\zeta\abs{s_1-s_2}^{\nu}=\zeta\abs{t_1-t_2}^{\nu}.
\end{split}
\end{equation*}
If $p_0=T_1=NT$, then we directly apply \eqref{e:TIMEHOLDER} and conclude from
the above considerations that
\begin{equation*}
S(p_j)x\to S(p_0)x\text{ as }j\to\infty\text{ for any }x\in\mathbf{B}.
\end{equation*}
We distinguish two cases.
If $K=\sup\{k_j\colon j\in\N\}<\infty$, then the elements $y_j$ belong to the finite set $\bigcup_{k=k_0}^{K}Q_k$
and we find a constant subsequence $y_{j_l}=y_0\in\bigcup_{k=k_0}^{K}Q_k\subseteq\mathbf{B}$.
It follows that 
$$x_{n_{j_l}}=S(p_{j_l})y_{j_l}=S(p_{j_l})y_0\to S(p_0)y_0\in\bigcup_{k=k_0}^{K}\bigcup_{p\in[NT,(N+1)T]}S(p)Q_k\subseteq\mathbf{E}\subseteq\mathbf{M}.$$
If $\sup\{k_j\colon j\in\N\}=\infty$, then there exists a~subsequence $k_{j_l}\to\infty$ such that
$$x_{n_{j_l}}=S(p_{j_l})y_{j_l}\in S(p_{j_l})S(k_{j_l}T)\mathbf{B}\subseteq S(k_{j_l}T)\mathbf{B}.$$
Since $\dist^{V}(S(k_{j_l}T)\mathbf{B},\mathbf{A})\to 0$ as $l\to\infty$, the sequence $x_{n_{j_l}}$
has a convergent subsequence to an element of $\mathbf{A}\subseteq\mathbf{M}$.

We now show that $\mathbf{A}\cup\mathbf{E}=\cl_{V}\mathbf{E}$.
Since $\cl_{V}\mathbf{E}$ is a compact set attracting all bounded sets, by
the minimality of the global attractor we have $\mathbf{A}\subseteq\cl_{V}\mathbf{E}$ and thus $\mathbf{A}\cup\mathbf{E}\subseteq\cl_{V}\mathbf{E}$.
For the converse inclusion, let $x\in\cl_{V}\mathbf{E}$. Then there exists a sequence
$x_j\in\mathbf{E}$ such that $x_j\to x$. Therefore, there are $k_j\geq k_0$
and $p_j\in[NT,(N+1)T]$ such that $x_j=S(p_j)y_j$ with some $y_j\in Q_{k_j}$.
Taking a subsequence if necessary, we assume that $p_j\to p_0\in[NT,(N+1)T]$.
Following the above arguments, if $K=\sup\{k_j\colon j\in\N\}<\infty$, then
$x_j$ has a convergent subsequence to an element of $\mathbf{E}$ and hence,  $x\in\mathbf{E}$.
If $\sup\{k_j\colon j\in\N\}=\infty$, then $x_j$ has a~convergent subsequence to an element
of $\mathbf{A}$, so $x\in\mathbf{A}$. We conclude that $x\in\mathbf{A}\cup\mathbf{E}$,
which completes the proof.
\end{proof}

We now use Theorem \ref{thm:SEM3} to formulate an existence result for exponential attractors based on the quasi-stability of the semigroup, see also \cite[Theorem 3.4]{sonner2012} and \cite[Theorem 3.4.7]{chueshov}. Corresponding results can be formulated for the other classes of semigroups considered in Sections \ref{sec:QS}--\ref{sec:Lad}.

\begin{thm}
Let $\{S(t)\colon t\geq 0\}$ be an asymptotically closed semigroup on a~complete metric space $(V,d)$, $T>0$ and let $\mathbf{B}\subseteq V$ be a bounded absorbing set. If the semigroup is quasi-stable on $\mathbf{B}$ at time $T$ with respect to a~compact seminorm $\mathfrak{n}_{Z}$ and parameters $(\eta,\kappa)$ and there exist $T_2>T_1\geq 0$, $\zeta>0$ and $\nu>0$ such that
\begin{equation*}
d(S(t_1)x,S(t_2)x)\leq\zeta\abs{t_1-t_2}^{\nu},\ t_1,t_2\in[T_1,T_2],\ x\in\mathbf{B},
\end{equation*}
then for any $\sigma\in(0,1-\eta)$ there exists an exponential attractor $\mathbf{M}\subseteq\mathbf{B}$ in $V$ (independent of $\nu$, $\zeta$, $T_2$) for the semigroup with rate of attraction $\xi\in(0,\frac{1}{T}\ln{\frac{1}{\eta+\sigma}})$, and its fractal dimension is bounded by
\begin{equation*}
\dimf^{V}(\mathbf{M})\leq\tfrac{1}{\nu}+\log_{\frac{1}{\eta+\sigma}}\mathfrak{m}_{Z}\left(\tfrac{\sigma}{2\kappa}\right).
\end{equation*}
We also have
$\mathbf{M}=\mathbf{A}\cup\mathbf{E}=\cl_{V}\mathbf{E}\subseteq\mathbf{B},$
where the set $\mathbf{E}$ is specified in Theorem \ref{thm:SEM3}. 
\end{thm}

\begin{proof}
We assume that $\mathbf{B}$ is positively invariant by Remark~\ref{rem:POSINV}.
The result is an immediate consequence of Theorems~\ref{thm:QUASI} and~\ref{thm:SEM3}, since the quasi-stability of the semigroup on a positively invariant set $\mathbf{B}$ at time $T$ with respect to a~compact seminorm $\mathfrak{n}_{Z}$ and parameters $(\eta,\kappa)$
implies that
\begin{equation}\label{e:LIPTSEMK}
d(S(kT)x,S(kT)y)\leq L^{k} d(x,y),\ x,y\in \mathbf{B},\ k\in\N,
\end{equation}
with some constant $L\geq 0$. In particular, the condition \eqref{e:LIPTSEMN} is satisfied.

Indeed, we show \eqref{e:LIPTSEMK} for $k=1$ by contradiction. For this purpose, suppose that for any $l\in\N$ there are $x_l, y_l\in \mathbf{B}$ such that
$$l d(x_l,y_l)<d(S(T)x_l,S(T)y_l).$$
Thus $x_l\neq y_l$ and by \eqref{e:DIMSEM2}
\begin{equation}\label{e:UNBSP}
l<\eta+\mathfrak{n}_{Z}(z_l),\ l\in\N,
\end{equation}
with $z_l=\frac{Kx_l-Ky_l}{d(x_l,y_l)}$.
Since by \eqref{e:DIMSEM1} we have $\norm{z_l}_{Z}\leq\kappa$, $l\in\N$, it follows from
the compactness of $\mathfrak{n}_{Z}$ that there exists a Cauchy subsequence $z_{l_j}$
with respect to $\mathfrak{n}_{Z}$. In particular, the sequence $\mathfrak{n}_{Z}(z_{l_j})$
is bounded which contradicts \eqref{e:UNBSP}.
Finally, \eqref{e:LIPTSEMK} follows by induction using the positive invariance of $\mathbf{B}$.
\end{proof}

To conclude, we mention the concept of non-autonomous exponential attractors for semigroups introduced in \cite{CaSo, sonner2012} and compare it with the notion of $T$-discrete exponential attractors. As in the latter case, the idea is to weaken the invariance property of exponential attractors replacing it by the positive invariance in the non-autonomous sense. 
\begin{defn}
A \emph{non-autonomous exponential attractor} for a semigroup $\{S(t)\colon t\geq 0\}$ on a metric space $(V,d)$ is a~non-autonomous  set $\mathbf{M}=\{\mathbf{M}(t)\colon t\geq 0\}\subseteq V$ such that $\mathbf{M}(t)$ is nonempty and compact and there exists $T>0$ such that $\mathbf{M}(t+T)=\mathbf{M}(t)$ for all $t\geq 0$.  Moreover, $\mathbf{M}(t)$ satisfies properties (ii) and (iii) in Definition \ref{defn:EXPAT} for all $t\geq 0$ with uniform constants $\xi$ and $\chi$ and the positive invariance (i) is replaced by 
\begin{itemize}
\item[(i'')] $\mathbf{M}$ is positively invariant in the non-autonomous sense, i.e., 
$$S(t)\mathbf{M}(s) \subseteq \mathbf{M}(t + s)\ \text{for all}\  t, s \geq 0.$$
\end{itemize}
\end{defn}

To construct  non-autonomous exponential attractors for semigroups defined on the time interval $[0,\infty)$, we only need the Lipschitz continuity of the semigroup in the phase space on a bounded absorbing set; the H\"older continuity in time is not required. Moreover, we obtain the same estimate for the fractal dimension as for $T$-discrete exponential attractors.

\begin{prop}
Let $\{S(t)\colon t\geq 0\}$ be an asymptotically closed semigroup on a complete metric space $(V,d)$ such that condition $(2)$ in Theorem \ref{thm:EXDEXPAT} holds with $T>0$ and a bounded absorbing set $\mathbf{B}\subseteq V$.
If the semigroup is Lipschitz continuous on $\mathbf{B}$, i.e., there exists $L_t>0$ such that 
$$d(S(t)x,S(t)y)\leq L_t d(x,y),\ x,y\in \mathbf{B},\ t\geq 0,$$
then there exists a non-autonomous exponential attractor  $\mathbf{M}= \{\mathbf{M}(t)\colon t\geq 0\}$ for the semigroup 
such that $\mathbf{M}(t+T)=\mathbf{M}(t)\subseteq\mathbf{B}$ for $t\geq 0$
and 
$$\dimf^{V}(\mathbf{M}(t))\leq\log_{\frac{1}{q}}h,\ t\geq 0.$$
Moreover, we have $\mathbf{A}\subseteq\mathbf{M_0}\subseteq\mathbf{M}(0)\subseteq\mathbf{B},$ where $\mathbf{A}$ is the global attractor and $\mathbf{M_0}$ is the $T$-discrete exponential attractor from Theorem \ref{thm:EXDEXPAT}.
\end{prop}

\begin{proof}
Let $\mathbf{M_0}=\cl_{V}\mathbf{E_0}$ be the $T$-discrete exponential attractor constructed in Theorem \ref{thm:EXDEXPAT}. We then obtain a non-autonomous exponential attractor  $\mathbf{M}= \{\mathbf{M}(t)\colon t\geq 0\}$ by setting
$$\mathbf{M}(t)=\cl_{V}(S(t)\mathbf{E_0}),\ t\in [kT,(k+1)T),\ k\in \N_0.$$
For details we refer to the proof of Theorem 3.15 in \cite{sonner2012}.
\end{proof}

\section{Supplementary observations}\label{sec:FC}

The notion of exponential attractors appeared in several hundreds of scientific papers and the existence of exponential attractors was proved for various problems originating from the Applied Sciences. Including examples for each construction method or a complete overview of these applications is beyond the scope of this paper. 
We refer the reader to monographs, survey articles and papers containing a~wide range of applications, like \cite{chueshov,cl2008,EFNT94, MiZe, CCM}. Nevertheless, we point out that some construction methods are more natural to use in certain situations and will mention a few classical problem classes as examples.  Semigroups satisfying the squeezing property and Ladyzhenskaya type semigroups (see Sections \ref{sec:SQ} and \ref{sec:Lad})  were originally designed for problems set in Hilbert spaces and exploit orthogonal projections onto finite-dimensional subspaces spanned by eigenfunctions of the main linear operator. Prominent examples are reaction diffusion equations and the 2D Navier-Stokes equations in bounded domains in an $L^2$- or $H^1$-setting, see e.g. \cite{EFNT94}. 
If an evolution equation is considered in a Banach space, which is not a Hilbert space, the more versatile (generalized) smoothing property (see Section \ref{sec:SP}) is particularly useful and can be applied, e.g. to problems generating semigoups in H\"older spaces or $L^p$-spaces, $p\neq2$, like in \cite{CaSo14, EfMiZe}. 
In fact, this method has been applied to a wide range of problems. For instance, semilinear parabolic equations in bounded domains, such as reaction-diffusion equations, Navier-Stokes type equations and Cahn-Hilliard type equations, satisfy the smoothing property without contraction mapping, i.e., $C\equiv0$ in \eqref{eq:contraction}, see e.g. \cite{EfMiZe,KoSo1, MiZe05}. On the other hand, semilinear damped wave equations in bounded domains or equations with memory satisfy the smoothing property with contraction term $C\not\equiv 0$, see e.g. \cite{CaSo14,CCM, GGP,KoSo}. 
The most general construction method based on quasi-stability (see Section \ref{sec:QS}) can be applied to parabolic equations in unbounded domains, to wave and plate type models with nonlinear and thermal damping or to lattice dynamical systems stemming from discretized parabolic problems, see e.g.  \cite{chueshov,cl2008,CC20,CC21,CC23}. 
Once again, we emphasize that the results in our paper provide a unifying framework for the construction of exponential attractors and allow to identify how the different construction methods are related, as outlined in the figures in the Introduction. Thus, in general, one can apply more than one method to construct an exponential attractor for a specific problem, but certain properties are easier to verify than others depending on the concrete situation, or lead to sharper estimates for the fractal dimension in a Hilbert space setting (see Sections \ref{sec:SQ} and \ref{sec:Lad}). 

Some approaches to construct exponential attractors and related notions used in the literature are still worth additional comments.  
J.~M\'alek and D.~Pra\v{z}\'ak introduced in \cite{MP02} the \emph{$\ell$-trajectory approach} considering  pieces of solutions on a time interval of given length $\ell>0$ with values in the phase space. The space of $\ell$-trajectories is a metric space, which is not necessarily complete.
The existence of a global attractor in the space of $\ell$-trajectories for such semigroups can be shown by Theorem~\ref{thm:existglobal} applying the Aubin-Lions-Dubinski compactness theorem. The latter result is also helpful to verify the covering condition~\eqref{e:CONDBALLSSEMB} by the methods presented in this paper exploiting the available compact embedding. For instance,  in \cite{MP02} the squeezing property was used to prove the existence of an exponential attractor in the space of $\ell$-trajectories. Then, using the H\"older continuity of the map which assigns to each $\ell$-trajectory its end point, one can obtain a global attractor and an exponential attractor for the semigroup in the original phase space.      
Sometimes authors also use known constructions in their own context, e.g., Ladyzhenskaya type semigroups from Section~\ref{sec:Lad} were applied in \cite{CH25} to construct $T$-discrete exponential attractors in Banach spaces for functional differential equations.

Theorem~\ref{thm:existglobal} provides different conditions that guarantee that a semigroup in a metric space $V$ possesses a global attractor $\mathbf{A}$ which are equivalent to the existence of a~nonempty bounded absorbing set combined with the \emph{asymptotic compactness} of the semigroup. A~semigroup is asymptotically compact if for any bounded subset $B\subseteq V$ such that $\gamma^{+}(B)=\{S(t)x\colon x\in B,\ t\geq 0\}$ is eventually bounded in~$V$, i.e.,
$S(t_0)\gamma^{+}(B)$ is bounded in $V$ for some $t_0\geq 0$, and for any sequences $t_k\geq 0$, $t_{k}\to\infty$ and $x_k\in B$, there exists a~convergent subsequence of $S(t_k)x_k$.  
There are many conditions in the literature that are equivalent to  the asymptotic compactness of the semigroup,
namely the flattening condition \cite{MWZ02,KL07,CLR13}, the double limes inferior condition \cite{Khanmamedov06}, \cite[Proposition 2.2.18]{chueshov}, and other conditions found e.g. in \cite{chueshov, Lad91}. Such conditions are frequently used as tools to prove the existence of the global attractor. However, to show existence of a $T$-discrete exponential attractor, in addition, one needs to verify the covering condition~\eqref{e:CONDBALLSSEMB}, see Theorem~\ref{thm:EXDEXPAT}.  

Note that if a semigroup on a bounded absorbing set $\mathbf{B}$ at time $T>0$ in a complete metric space $(V,d)$  is quasi-stable, then $S(T)$ is an \emph{$\alpha$-contraction} with $\eta\in[0,1)$. This means that the semigroup possesses a  bounded absorbing set $\mathbf{B}$ and for any nonempty subset $B$ of $\mathbf{B}$ such that $\gamma^{+}(B)\subseteq B$ we have
$$\alpha_{V}(S(T)B)\leq\eta\alpha_{V}(B),$$
where $\alpha_{V}(\cdot)$ denotes the Kuratowski measure of noncompactness. 
This implies that
$$\alpha_{V}(S(t)B)\to 0\ \text{ as}\ t\to\infty$$
for any bounded subset $B$ such that $\gamma^{+}(B)\subseteq B$, and hence,  the semigroup is asymptotically compact.  
For maps which are $\alpha$-contractions, the authors in \cite{EFK98} used, however, the squeezing property from Section~\ref{sec:SQ} to construct exponential attractors.

Finally, we observe that the global attractor $\mathbf{A}$ may itself attract all bounded subsets exponentially. Sufficient conditions for this property were formulated e.g. in \cite{CCH11}. Although the authors did not require the finite fractal dimension of the attractor, it is clear that in such a situation one may employ classical methods to prove the finite fractal dimension of the global attractor, using e.g. exponential decay of the volume element or the Lyapunov exponents \cite[Theorems V.3.2, V.3.3]{temam} or the quasi-stability of the semigroup on the global attractor~$\mathbf{A}$ \cite[Theorem 3.4.5]{chueshov}.

\section*{Acknowledgements}
The authors thank the referee for the valuable comments and 
Jan Cholewa for inspiring discussions on the notion of quasi-stability. The first author would like to kindly thank colleagues at the Radboud University in Nijmegen for their hospitality during his visit at this institution. The second author is grateful for the great hospitality of colleagues at the University of Silesia during her stay in Katowice.

\end{document}